\documentclass[a4paper,12pt]{article}
\usepackage{lineno}
\usepackage{amsfonts}
\usepackage{bbold}
\usepackage{xy}
\usepackage{pstricks}
\usepackage{amssymb}

\begin{document}

\newtheorem{theorem}{Theorem}[section]

\newtheorem{corollary}[theorem]{Corollary}
\newtheorem{definition}[theorem]{Definition}
\newtheorem{conjecture}[theorem]{Conjecture}
\newtheorem{question}[theorem]{Question}
\newtheorem{lemma}[theorem]{Lemma}
\newtheorem{proposition}[theorem]{Proposition}
\newtheorem{example}[theorem]{Example}
\newtheorem{problem}[theorem]{Problem}
\newenvironment{proof}{\noindent {\bf
Proof.}}{\rule{3mm}{3mm}\par\medskip}
\newcommand{\remark}{\medskip\par\noindent {\bf Remark.~~}}
\newcommand{\pp}{{\it p.}}
\newcommand{\de}{\em}
 \newcommand{\LMA}{{\it Linear and Multilinear Algebra},  }
\newcommand{\JEC}{{\it Europ. J. Combinatorics},  }
\newcommand{\JCTB}{{\it J. Combin. Theory Ser. B.}, }
\newcommand{\JCT}{{\it J. Combin. Theory}, }
\newcommand{\JGT}{{\it J. Graph Theory}, }
\newcommand{\ComHung}{{\it Combinatorica}, }
\newcommand{\DM}{{\it Discrete Math.}, }
\newcommand{\ARS}{{\it Ars Combin.}, }
\newcommand{\SIAMDM}{{\it SIAM J. Discrete Math.}, }
\newcommand{\SIAMADM}{{\it SIAM J. Algebraic Discrete Methods}, }
\newcommand{\SIAMC}{{\it SIAM J. Comput.}, }
\newcommand{\ConAMS}{{\it Contemp. Math. AMS}, }
\newcommand{\TransAMS}{{\it Trans. Amer. Math. Soc.}, }
\newcommand{\AnDM}{{\it Ann. Discrete Math.}, }
\newcommand{\NBS}{{\it J. Res. Nat. Bur. Standards} {\rm B}, }
\newcommand{\ConNum}{{\it Congr. Numer.}, }
\newcommand{\CJM}{{\it Canad. J. Math.}, }
\newcommand{\JLMS}{{\it J. London Math. Soc.}, }
\newcommand{\PLMS}{{\it Proc. London Math. Soc.}, }
\newcommand{\PAMS}{{\it Proc. Amer. Math. Soc.}, }
\newcommand{\JCMCC}{{\it J. Combin. Math. Combin. Comput.}, }
\newcommand{\GC}{{\it Graphs Combin.}, }

\title{ The Minimum Spectral Radius of Graphs with  the  Independence Number
\thanks{ This work is supported by National Natural Science
Foundation of China (No: 11271256).
}}
\author{   Ya-Lei Jin  and Xiao-Dong Zhang \\
{\small Department of Mathematics, and MOE-LSC,}\\
{\small Shanghai Jiao Tong University} \\
{\small  800 Dongchuan road, Shanghai, 200240, P.R. China}\\
{\small Email:  xiaodong@sjtu.edu.cn}
 }
\date{}
\maketitle
 \begin{abstract}
 In this paper,  we investigate  some properties of the Perron vector of connected graphs. These results are used to characterize that   all extremal connected graphs
    with having the minimum (maximum) spectra radius among all connected graphs of order $n=k\alpha$ with the independence number $\alpha$, respectively.
    
 \end{abstract}

{{\bf Key words:} Perron vector;
Independent number; Graph; Spectra radius.
 }

      {{\bf MSC:} 05C50, 05C35}
\vskip 0.5cm

\section{Introduction}
Throughout this paper,  we always consider  simple graphs.
  Let $G = (V(G),~E(G))$ be a simple graph with vertex set
$V(G)=\{v_1, \dots, v_n\}$ and edge set $E(G)$.  Let $A(G)=(a_{ij})$
be the $(0,1)$ {\it  adjacency matrix} of $G$ with $a_{ij}=1$ for
$v_i\sim v_j$ and $0$ otherwise, where $``\sim"$ stands for the adjacency relation.  
The largest eigenvalue of $A(G)$ is called {\it spectra radius} of $G$, denoted by $\lambda(G)$. The {\it independent number} (also the stability number) of $G$, denoted by $\alpha(G)$, is the size of the subset of $V(G)$, such that every pair vertices of this subset are not adjacent.

A classical Tur\'{a}n \cite{Turan1941} theorem for the independence number stated that the Tur\'{a}n graph $T_{n,\alpha}$ which consists of $\alpha$ disjoint balanced cliques  is a unique graph having the minimum size among all graphs  of order $n$ and the independence number $\alpha$.
  Since the Tur\'{a}n graph is disconnected, Ore \cite{Ore1962} raised how  to determine the minimum  number of edges of all connected graphs with order $n$ and independence number $\alpha$.  Recently, this problem was settled independently by Bougard and Joret \cite{boug2008}, and by Gitler and Valencia \cite{Git}.
  In spectral extremal graph theory, Guiduli \cite{guiduli1998}  and Nikiforov \cite{nikiforov2007}  independently proved a spectral extremal Tur\'{a}n theorem.
  Hence the Tur\'{a}n graph is a unique graph having the minimum spectral radius among all graphs of order $n$ and the independence number $\alpha$.  Moreover,
  Stevanovi\'{c} and Hansen \cite{Stevanovic2008} determined all extremal graphs with minimum spectral radius among all connected graphs of order $n$ and the clique number $\omega$.  It is natural to raise the following problem.
  \begin{problem}\label{prob1}
  Determine the minimum (maximum) spectral radius of matrices (for example, the adjacency, Laplacian, Signless, Distance matrices, etc) associated with a  connected graph of order $n$ and the independence number $\alpha$. Moreover, characterize all extremal graphs which attain the bound
   \end{problem}
   Recently, Xu et al.\cite{xu2009} characterized all extremal graphs with minimum spectral radius among all connected graphs of order $n$ and independence number $\alpha \in \{1,2,\lceil\frac{n}{2}\rceil,\lceil\frac{n}{2}\rceil+1,n-3,n-2,n-1\}$.  Du and Shi \cite{du2012}    proposed the following conjecture
    \begin{conjecture}\label{conjecture1}
    The graph obtained from a path of order $\alpha$ by blowing up each vertex to a clique of order $k$ minimizes the spectral radius among all connected graphs of $k\alpha$ with independence number $\alpha$.
    \end{conjecture}
They proved this conjecture is true for  $\alpha=3,4$.
Motivated by Conjecture~\ref{conjecture1} and the above results, we study some properties of extremal graphs having the minimum spectral radius. Before stating our main results, we need some notations.  Let $ \mathcal{G}_{n,\alpha}$ be the set of all connected graphs of order $n$ with independence number $\alpha$ and let $\mathcal{T}_{n,\alpha}$ be the set of all graphs of order $n$ obtained from a tree of order $\alpha$ by replacing each vertex to a clique of order ${\lfloor\frac{n}{\alpha}\rfloor}$ or $ {\lceil \frac{n}{\alpha}\rceil}$.
 A graph $G$ of order $n$ with independence number $\alpha$ and $n\ge 3\alpha$ is called {\it clique path} and denoted by $P_{n,\alpha}$, if $G$  is obtained from a path of order $\alpha$ by replacing each vertex to a clique of order ${\lfloor\frac{n}{\alpha}\rfloor}$ or $ {\lceil \frac{n}{\alpha}\rceil}$ such that there are $2\alpha-2$ cut vertices. A graph $G$  of order $n$ with independence number $\alpha$ is called {\it clique star} and denoted by $S(n,\alpha)$, if
   $G$  is obtained from star $K_{1,\alpha-1}$ by replacing each vertex to
   a clique of order ${\lfloor\frac{n}{\alpha}\rfloor}$ or $ {\lceil \frac{n}{\alpha}\rceil}$ such that there are exactly $\alpha$ cut vertices. If $n=k\alpha$, then there exists a unique clique path and a unique star in $ \mathcal{G}_{n,\alpha}$. But, if $n\neq k\alpha$,  clique paths and clique star in $ \mathcal{G}_{n,\alpha}$ are  not unique.
  Moreover,
  Let
$$\lambda_{n,\alpha}=\min\{\rho(G): G\ \mbox{is a connected graph of order}\ n \ \mbox{with independence number}\ \alpha \},$$
$${\Lambda}_{n,\alpha}=\max\{\rho(G): G\ \mbox{is a connected graph of order}\ n \ \mbox{with independence number}\ \alpha \}$$
 The main results of this paper are states as follows.
 \begin{theorem}\label{ess} Fixed $\alpha$. Then
$$\lim_{n\rightarrow\infty}\frac{\lambda_{n,\alpha}}{n}=\frac{1}{\alpha}.$$
\end{theorem}
  \begin{theorem}\label{T2} If $n=k\alpha$ and $k> \frac{17\alpha+15}{8}$, then
  $P_{n,\alpha}$ is the only graph having the minimum spectral radius in $\mathcal{G}_{n,\alpha}$. In other words, for  any $G\in \mathcal{G}_{n,\alpha}$, $\lambda(G)\ge \lambda(P_{n,\alpha})$ with equality if and only if $G$ is $P_{n,\alpha}$.
  \end{theorem}

 \begin{theorem}\label{T3}
If $n=k\alpha$, then the clique star is the only graph having the maximum spectral radius in $\mathcal{G}_{n,\alpha}$. In other words, for  any $G\in \mathcal{G}_{n,\alpha}$, $\lambda(G)\le \lambda(S_{n,\alpha})$ with equality if and only if $G$ is $S_{n,\alpha}$.
\end{theorem}

{\bf Remark} Theorem~1 may be regarded as a spectral form of the well-known Erd\H{o}s-Stone-Simonovits theorem~\cite{bondy2008}, while Theorem 2 proves that  Conjecture~\ref{conjecture1} is true under minor conditions. The rest of this paper is organized as follows. In Section 2, we present a proof of Theorem~\ref{ess} and relative results. In Section 3, we present proofs of Theorems~\ref{T2} and \ref{T3}.

 \section{A spectral Erd\H{o}s-Stone-Simonovits type theorem}

 In order to prove Theorem~\ref{ess}, we need the following lemmas.
 \begin{theorem}\cite{du2012}\label{T4}
(1). Every nonbipartite triangle free graph of order $n$ has at most $1+\frac{(n-1)^2}{4}$;\\
(2). If $G$ is a $K_{r+1}$-free graph of order $n$ with chromatic number at least $r+1>2$, then $|E(G)|\leq \frac{(r-1)n^2}{2r}-\frac{n}{2r}+\frac{17}{16}-\frac{1}{8r}$.
\end{theorem}
\begin{lemma}\label{L1}
If $n=k\alpha+t,\alpha>1$, then
\begin{displaymath}
\lambda_{n,\alpha} \leq \left\{ \begin{array}{ll}
k-1+\frac{2}{k-1}, & \textrm{ $t=0$}\\
k+\frac{2}{k}, & \textrm{$1\leq t<\alpha.$ }
\end{array} \right.
\end{displaymath}
\end{lemma}
\begin{proof}
  If $t=0,$  $P_{n,\alpha}$ is just as the following $Fig.1$.

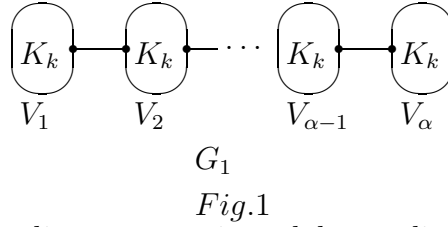
\begin{picture}(0, 20)
\put(45,6){$K_k$}
\put(48,8){\oval(8,12)}
\put(52,8){\line(1,0){7}}
\put(52,8){\circle*{1}}
\put(59,8){\circle*{1}}
\put(45,-2){$V_1$}
\put(60,6){$K_k$}
\put(63,8){\oval(8,12)}
\put(67,8){\line(1,0){4}}
\put(67,8){\circle*{1}}
\put(60,-2){$V_2$}
\put(72,8){$\ldots$}
\put(80,6){$K_k$}
\put(83,8){\oval(8,12)}
\put(87,8){\line(1,0){7}}
\put(87,8){\circle*{1}}
\put(94,8){\circle*{1}}
\put(80,-2){$V_{\alpha-1}$}
\put(95,6){$K_k$}
\put(98,8){\oval(8,12)}
\put(95,-2){$V_{\alpha}$}
\put(68,-8){$G_1$}
\put(68,-14){$Fig.1$}
\end{picture}\\ \\  \\
Let $A(P_{n,\alpha}),D(P_{n,\alpha})$ be the adjacency matrix and degree diagonal matrix of $P_{n,\alpha}$. It is easy to see that $A(P_{n,\alpha})$ and $D(P_{n,\alpha})^{-1}A(P_{n,\alpha})D(P_{n,\alpha})$ have the same eigenvalues.
 If $v\in V_1$ or $V_\alpha$, then the sum of the row corresponding to $v$ in $D(P_{n,\alpha})^{-1}A(P_{n,\alpha})D(P_{n,\alpha})$ is at most $\max\{k-1+\frac{1}{k-1},\frac{k-1}{k}(k-1)+1\}=k-1+\frac{1}{k-1}$.
 If $v\in V_i,2\leq i\leq \alpha$, then the sum of the row corresponding to $v$ in the matrix $D(P_{n,\alpha})^{-1}A(P_{n,\alpha})D(P_{n,\alpha})$ is at most $\max\{k-1+\frac{2}{k-1},k-1+\frac{2}{k}\}=k-1+\frac{2}{k-1}$.
     Hence  $\lambda_{n,\alpha}<\lambda(P_{n,\alpha})<k-1+\frac{2}{k-1}$. If $1\le t<\alpha$, it is easy to see that $P_{n,\alpha}$ is a subgraph of $P_{n+\alpha-t,\alpha}$. Hence  $\lambda_{n,\alpha}\leq \lambda(P_{n,\alpha})<\lambda(P_{n+\alpha-t,\alpha})<k+\frac{2}{k}$, since $n+\alpha-t=(k+1)\alpha$. This completes the proof.
\end{proof}
 We are ready to prove Theorem~\ref{ess}.

\begin{proof} Firstly we show that the upper limit of $\frac{\lambda_{n,\alpha}}{n}$ is $\frac{1}{\alpha}$.  Considering the following two cases:\\
{\bf Case 1:} $n=k\alpha$. By the Lemma \ref{L1}, $\lambda_{n,\alpha}<k-1+\frac{2}{k-1}$, which implies  $$\overline{\lim_{n\rightarrow\infty} }\frac{\lambda_{n,\alpha}}{n}\leq \overline{\lim_{n\rightarrow\infty}}\frac{k-1+\frac{2}{k-1}}{n}=\frac{1}{\alpha}.$$
{\bf Case 2:}  $n= k\alpha+t, 0<t<\alpha$. By the Lemma \ref{L1}, $\lambda_{n,\alpha}<k+\frac{2}{k}$, which implies $$\overline{\lim_{n\rightarrow\infty} }\frac{\lambda_{n,\alpha}}{n}\leq \overline{\lim_{n\rightarrow\infty}}\frac{k+\frac{2}{k}}{n}=\frac{1}{\alpha}.$$
Next we will show that the lower limit of $\frac{\rho(n,\alpha)}{n}$ is also $\frac{1}{\alpha}$. Suppose $$\underline{\lim}_{n\rightarrow\infty}\frac{\rho(n,\alpha)}{n}=\frac{1}{\alpha}-2\epsilon,0<2\epsilon\leq \frac{1}{\alpha} .$$
Then there exists an increasing sequence $\{n_i\}_{i=1}^\infty$, and a sequence of graphs $\{G_i\}_{i=1}^\infty$, where $G_i$ is a graph of order  $n_i$ with size $m_i$ such that $\lambda(G_i)\leq  (\frac{1}{\alpha}-\epsilon)n_i$. Since $\lambda(G_i)\geq \frac{2m_i}{n_i}$,  we have
\begin{eqnarray*}
|E(G_i^c)|&\geq& \frac{n_i(n_i-1)}{2}-\frac{1}{2}(\frac{1}{\alpha}-\epsilon)n_i^2\\
&=&\frac{\alpha-1}{2\alpha}n_i^2+\frac{\epsilon n_i^2}{2}-\frac{n_i}{2}\\
&>&\frac{\alpha-1}{2\alpha}n_i^2-\frac{n_i}{2\alpha}+\frac{17}{16}-\frac{k}{8n_i},
\end{eqnarray*}
 for $i$ is large enough. By Lemma~\ref{T4}, the chromatic number of $G_i^c$ is at most $\alpha$ if $i$ is large enough. Then $G_i$ contains a clique of order $\lceil\frac{n}{\alpha}\rceil$, which imples $\lambda(G_i)>\lceil\frac{n_i}{\alpha}\rceil-1$. Hence
  $\frac{\lambda(G_i)}{n_i}\rightarrow \frac{1}{\alpha}$ as $n_i$ tends to infinity. It is a contradiction with $\alpha(G_i)\leq  (\frac{1}{\alpha}-\epsilon)n_i$ for all $i$. So $$\underline{\lim}_{n_i\rightarrow\infty}\frac{\rho(n,\alpha)}{n}=\frac{1}{\alpha}.$$
This completes the proof.
\end{proof}
 {\bf Remark} It follows from Theorem~\ref{ess} that a graph of order $n=k\alpha$ with spectral radius $\lambda(G)\le (\frac{1}{\alpha}-\epsilon)n$ for positive number $\varepsilon>0$ has an independent set with at least $\alpha$. It is an interesting question to count how many such independent sets? Denote by $i_s(G)$  the number of $s$-independent set of $G$ and $k_s(G)$ for the number of $s$-clique of $G$. It is easy to see that $k_s(G)=i_s(G^c)$.  Bollob\'{a}s and Nikiforov \cite{bela2007} gave a lower bound for $k_{r+1}(G)$ in terms of  spectral radius.
\begin{lemma}\cite{bela2007}\label{bv}
For any graph $G$ of order $n$, and $r>1$,
$$k_{r+1}(G)\geq \left(\frac{\lambda(G)}{n}-1+\frac{1}{r}\right)\frac{r(r-1)}{r+1}\left(\frac{n}{r}\right)^{r+1}.$$
\end{lemma}
By using the above Lemma, we present a lower bound for $i_s(G)$.
\begin{theorem}\label{innu}
Let  $G$  be a simple graph of order $n$ and $\alpha$ be an positive integer. If  $\lambda(G)\leq \frac{n}{\alpha}$, then
$$i_\alpha(G)\ge \left(\frac{1}{\alpha(\alpha-1)}-\frac{1}{n}\right)\frac{(\alpha-1)
(\alpha-2)}{\alpha}\left(\frac{n}{\alpha-1}\right)^\alpha.$$
\end{theorem}
\begin{proof}
Since $\frac{2m(G)}{n}\leq \lambda(G), $  we have  $m(G)\leq  \frac{n^2}{2\alpha}$, which implies $m(G^c)\geq  \frac{\alpha-1}{2\alpha}n^2-\frac{n}{2}$. So  $\lambda(G^c)\geq \frac{\alpha-1}{\alpha}n-1$. By the Theorem \ref{bv}, we can get
\begin{eqnarray*}
i_\alpha(G)= k_\alpha(G^c)&\geq& \left(\frac{\alpha-1}{\alpha}-\frac{1}{n}-1+\frac{1}{\alpha-1}\right)\frac{(\alpha-1)(\alpha-2)}{\alpha}\left(\frac{n}{\alpha-1}\right)^\alpha\\
&=&\left(\frac{1}{\alpha(\alpha-1)}-\frac{1}{n}\right)\frac{(\alpha-1)(\alpha-2)}{\alpha}\left(\frac{n}{\alpha-1}\right)^\alpha.
\end{eqnarray*}
This completes the proof.
\end{proof}
{\bf Remark} From  theorem~\ref{innu}, $i_s(G)$ is about  $O(n^\alpha)$ if $\lambda(G)\leq \frac{n}{\alpha}$.

\section{Proof of Theorems~\ref{T2} and \ref{T3}}

In order to prove Theorems\ref{T2} and \ref{T3}, we need some lemmas.

\begin{lemma}\label{Z}
Let $n=k\alpha$ and $k> \frac{17\alpha+15}{8}$. If a connected graph $G$ has the minimum spectra radius among all graphs in $\mathcal{G}_{n,\alpha}$, then $G$ has to be in $\mathcal{T}_{n,\alpha}$.
\end{lemma}
\begin{proof}
 By Lemma~\ref{L1},   $\lambda(G)=\lambda_{n,\alpha}\le k-1+\frac{2}{k-1}$ and $G$ does not contain $K_{k+1}$. Further, we claim that the chromatic number of $G^c$ is $\alpha$. Suppose that the chromatic number of $G^c$ is at least $\alpha +1$. By Lemma~\ref{T4},
\begin{eqnarray*}
|E(G)|&\geq& \frac{n(n-1)}{2}-\frac{(\alpha-1)n^2}{2\alpha}+\frac{n}{2\alpha}-\frac{17}{16}+\frac{1}{8\alpha}\\
&=&\frac{kn}{2}-\frac{n-k}{2}-\frac{17}{16}+\frac{k}{8n}\\
&=&\frac{(k-1)n}{2}+\frac{k}{2}-\frac{17}{16}+\frac{k}{8n}.
\end{eqnarray*}
Then by $k> \frac{17\alpha+15}{8}$, we have
$$
 \lambda(G) \geq  \frac{2|E(G)|}{n}
\ge  k-1+\frac{1}{\alpha}-\frac{17}{8n}+\frac{k}{4n^2}
 > k-1+\frac{2}{k-1}.
 $$
  Hence the chromatic number of $G^c$ is $\alpha$, i.e., $G^c$ is an $\alpha$-partite graph. Assume the parts of $G^c$ are $V_1,V_2,...,V_\alpha$. Since $G$ does not contain $K_{k+1}$ and $n=k\alpha$, $|V_1|=|V_2|=...=|V_\alpha|=k$. Moreover, the induced subgraph by $V_i\bigcup V_j$ ($i\neq j$) is not complete bipartite, since $G$ is connected. Note that the spectral radius of a connected graph is an  strictly increasing function with respect to adding an edge.  Hence $G$ has to be in  $\mathcal{T}_{n,\alpha}$.
\end{proof}

 \begin{lemma}\label{T1}
 Let $G$ be a non-bipartite connected graph of order $n$ and  $x=(x_1,x_2,...,x_n)^T$ be the Perron vector of $A(G)$.  If $\sigma_s(v_i)$ is the number of all closed walks of length $k$ containing vertex $v_i$, $ i=1, \cdots,n$, then
  $$ \lim_{s\rightarrow\infty}\frac{\sigma_s(v_i)}{\sigma_s(v_j)}\ge 1$$
  if and only if $x_i\ge x_j$.
\end{lemma}
\begin{proof} By spectral decomposition theorem, there exist eigenvalues $\lambda_2,\cdots, \lambda_n$ and corresponding orthogonal eigenvectors $\xi_2,\cdots, \xi_n$ such that
 $$A=\lambda(A)xx^T+\lambda_2\xi_2\xi_2^T+...+\lambda_n\xi_n\xi_n^T.$$
 Then
 $$A^s=\lambda(A)^skxx^T+\lambda_2^s\xi_2\xi_2^T+...+\lambda_n^s\xi_n\xi_n^T,$$
  Let $e_i$ be the column vector whose the $i$-th entry is 1 and 0 otherwise, $i=1,\dots, n$.
    Then $\sigma_s(v_i)=e_i^TA^se_i $.  Moreover $\lambda(G)>|\lambda_i|$ for $i=2,\dots, n$ since $G$ is non-bipartite and connected.
    Hence
    \begin{eqnarray*}
    \lim_{s\rightarrow\infty}\frac{\sigma_s(v_i)}{\sigma_s(v_j)}&=&
    \lim_{s\rightarrow\infty}\frac{e_i^TA^se_i}{e_j^TA^se_j}\\
    &=&\lim_{s\rightarrow\infty}\frac{e_ixx^Te_i+\frac{\lambda_2^s}{\lambda(A)^s}e_i^T\xi_2\xi_2^Te_i
    +\dots+\frac{\lambda_n^s}{\lambda(A)^s}e_i^T\xi_n\xi_n^Te_i}
    {e_jxx^Te_j+\frac{\lambda_2^s}{\lambda(A)^s}e_j^T\xi_2\xi_2^Te_j
    +\dots+\frac{\lambda_n^s}{\lambda(A)^s}e_j^T\xi_n\xi_n^Te_j}\\
    &=& \frac{x_i}{x_j}.\end{eqnarray*}
  This completes the proof.
\end{proof}

\begin{lemma}\label{lambda=k-1} Let $G$  be a graph of  order $n=k\alpha$ with the independence number $\alpha$.
 Then $\lambda(G)\geq k-1$ equality if and only if $G$ is union of the number $\alpha$  complete graphs $K_k$.
\end{lemma}
\begin{proof} If $k=1$ or $2$, the assertion follows from \cite{xu2009}. Assume that $k\ge 3$. Suppose that $\lambda(G)\leq k-1$. Then the size $m(G)$ of $G$ at most $ \frac{\lambda(G) n}{2}\le \frac{n-\alpha}{2\alpha}n$, which implies  the size  $m(G^c)$ of $G^c$ at least $ \frac{\alpha-1}{2\alpha}n^2$. We have the claim that  the chromatic number of $G^c$ is
 $\alpha$. In fact, if the chromatic number of $G^c$ is
 at least $\alpha+1$, then  by Theorem \ref{T4}, $m(G^c)\leq \frac{(\alpha-1)n^2}{2\alpha}-\frac{n}{2\alpha}+\frac{17}{16}-\frac{1}{8\alpha}\leq \frac{(\alpha-1)n^2}{2\alpha}-\frac{3\alpha}{2\alpha}+\frac{17}{16}-\frac{1}{8\alpha}< \frac{(\alpha-1)n^2}{2\alpha}$. This is a contradiction. Thus  $G^c$ is a $\alpha$-partite graph. Moreover,  suppose the parts of $G^c$ are $V_1,V_2,...,V_\alpha$. Then $|V_1|=|V_2|=...=|V_\alpha|=k$, since $G$ can not contain $K_{k+1}$. Hence $G$ is union of the number $\alpha$  complete graphs $K_k$. This completes the proof.
\end{proof}

\begin{lemma}\label{L2}
Let $n=k\alpha>2\alpha$ and $G\in  \mathcal{G}_{n,\alpha}$  be  a graph obtained by joining an edge from a non-cut vertex of a graph $H\in \mathcal{G}_{n-k(l+p), \alpha-(l+p)} $ and a non-cut vertex of $P_{k(l+p),l+p}$ (see $Fig.2$). Let  $G'$ be the graph obtained from $G$ by deleting the edge $v_3v_4$ and adding edge $v_1v_4$.  If $H$ contains a copy of $P_{kl, l}$ whose vertex set does not contains $v_1$, then $\lambda(G')> \lambda(G)$.\\ \\
\begin{picture}(60, 18)
\put(22,6){$K_k$}
\put(22,16){$V_{l+p}$}
\put(25,8){\oval(8,12)}
\put(29,8){\line(1,0){4}}
\put(29,8){\circle*{1}}
\put(34,8){$\ldots$}
\put(42,6){$K_k$}
\put(42,16){$V_{p+1}$}
\put(45,8){\oval(8,12)}
\put(49,8){\line(1,0){7}}
\put(49,8){\circle*{1}}
\put(56,8){\circle*{1}}
\put(52,10){$v_3$}
\put(49,5){$v_4$}
\put(64.5,10){$v_2$}
\put(57,6){$K_k$}
\put(57,16){$V_{p}$}
\put(60,8){\oval(8,12)}
\put(72,6){$K_k$}
\put(72,16){$V_{p-1}$}
\put(75,8){\oval(8,12)}
\put(64,8){\line(1,0){7}}
\put(79,8){\line(1,0){4}}
\put(64,8){\circle*{1}}
\put(71,8){\circle*{1}}
\put(79,8){\circle*{1}}
\put(84,8){$\ldots$}
\put(92,6){$K_k$}
\put(92,16){$V_1$}
\put(95,8){\oval(8,12)}
\put(99,8){\line(1,0){7}}
\put(99,8){\circle*{1}}
\put(106,8){\circle*{1}}
\put(102,10){$v_1$}
\put(101,22){$.............................$}
\put(101,-2.5){$\vdots$}
\put(101,2){$\vdots$}
\put(101,6){$\vdots$}
\put(101,10){$\vdots$}
\put(101,14){$\vdots$}
\put(101,18){$\vdots$}
\put(101,-2.5){$............................$}
\put(133,-2.5){$\vdots$}
\put(133,2){$\vdots$}
\put(133,6){$\vdots$}
\put(133,10){$\vdots$}
\put(133,14){$\vdots$}
\put(133,18){$\vdots$}
\put(107,6){$K_k$}
\put(107,16){$V_0$}
\put(110,8){\oval(8,12)}
\put(112,8){\line(1,0){7}}
\put(112,10){\line(1,0){7}}
\put(112,6){\line(1,0){7}}
\put(124,8){\circle{15}}
\put(125,18){$H$}
\put(135,6){$G$}
\end{picture}\\ \\ \\
\begin{picture}(60, 20)
\put(42,-4){$K_k$}
\put(42,6){$V_{l+p}$}
\put(45,-2){\oval(8,12)}
\put(49,-2){\line(1,0){7}}
\put(49,-2){\circle*{1}}
\put(56,-2){\circle*{1}}
\put(57,-4){$K_k$}
\put(57,6){$V_2$}
\put(60,-2){\oval(8,12)}
\put(64,-2){\line(1,0){4}}
\put(64,-2){\circle*{1}}
\put(69,-2){$\ldots$}
\put(77,-4){$K_k$}
\put(77,6){$V_{p+1}$}
\put(80,-2){\oval(8,12)}
\put(84,-2){\line(2,3){7}}
\put(84.5,-4.5){$v_4$}
\put(84,-2){\circle*{1}}
\put(41,18){\circle*{1}}
\put(37,20){$v_3$}
\put(49.5,20){$v_2$}
\put(42,16){$K_k$}
\put(42,26){$V_{p}$}
\put(45,18){\oval(8,12)}
\put(57,16){$K_k$}
\put(57,26){$V_{p-1}$}
\put(60,18){\oval(8,12)}
\put(49,18){\line(1,0){7}}
\put(64,18){\line(1,0){4}}
\put(49,18){\circle*{1}}
\put(56,18){\circle*{1}}
\put(64,18){\circle*{1}}
\put(69,18){$\ldots$}
\put(77,16){$K_k$}
\put(77,26){$V_1$}
\put(80,18){\oval(8,12)}
\put(84,18){\line(2,-3){7}}
\put(84,18){\circle*{1}}
\put(91,8){\circle*{1}}
\put(86,22){$.............................$}
\put(86,-2.5){$\vdots$}
\put(86,2){$\vdots$}
\put(86,6){$\vdots$}
\put(86,10){$\vdots$}
\put(86,14){$\vdots$}
\put(86,18){$\vdots$}
\put(86,-2.5){$............................$}
\put(118,-2.5){$\vdots$}
\put(118,2){$\vdots$}
\put(118,6){$\vdots$}
\put(118,10){$\vdots$}
\put(118,14){$\vdots$}
\put(118,18){$\vdots$}
\put(87,8){$v_1$}
\put(92,6){$K_k$}
\put(92,16){$V_0$}
\put(95,8){\oval(8,12)}
\put(97,8){\line(1,0){7}}
\put(97,10){\line(1,0){7}}
\put(97,6){\line(1,0){7}}

\put(109,8){\circle{15}}
\put(109,18){$H$}
\put(120,6){$G'$}
\put(69,-16){$Fig.2$}
\end{picture}\\ \\
\end{lemma}
\begin{proof}
Let $x=(x(u), u\in V(G))^T$ be the Perron vector of $G\in \mathcal{G}(n,\alpha)$ and let
$\mathcal{W}(s,v_i)$ be the set of all closed walks of length $s$ containing $v_i$, $i=1,2,3$. We claim that there exists an injective mapping $\varphi$ from $ \mathcal{W}(s,v_3)$ to $\mathcal{W}(s,v_2)$. In fact, if $W$ is a closed walk of length $s$ containing $v_3$ and $v_2$, let $\varphi(W)=W$. If $W$ is a closed walk of length $s$ containing $v_3$ and no $v_2$, then there exists a corresponding  closed walk $W^{\prime}=\varphi(W)$ of length $s$ containing $v_2$ and no $v_3$ in the subgraph $P_{k(2l+1),2l+1} $ in $G$. Hence $\sigma(s, v_3)\le \sigma(s,v_2)$.
Similarly,  there exists an injective mapping $ \phi$ from
$ \mathcal{W}(s,v_2)$ to $\mathcal{W}(s,v_1)$, which implies $\sigma(s, v_2)\le \sigma(s,v_1)$.  By Lemma~\ref{T1},
$x(v_3)\leq x(v_2)\leq x(v_1)$. Hence
  $\lambda(G)=x^TA(G)x=x^TA(G')x-2(x(v_1)-x(v_3))x(w)\le x^TA(G')x \leq \lambda(G')$.
  Moreover, if $\lambda(G)= \lambda(G')$, then  $x(v_1)=x(v_3)$ and $x$ is an eigenvector of $A(G')$. But it is impossible. Therefore $\lambda(G)<\lambda(G')$.
\end{proof}
\begin{lemma}\label{L5}
Let $n=k\alpha>2\alpha$ and $G\in  \mathcal{G}_{n,\alpha}$ be a graph with two vertices $u$ and $v$ which are in clique of order $k$, if $u$ is adjacent with $u_1$,$u_2$,..., $u_t$ which belong to $t$ vertex disjoint clique paths $P_{kl_1,l_1}$,$P_{kl_2,l_1}$,..., $P_{kl_t,l_t}(t>1)$ respectively, and $d_G(u)-t=d_G(v)=k-1$. Let $G'$ be the graph obtained from $G$ by deleting the edge $u_{1}u$ and adding edge $u_{1}v$. Then $\lambda(G')< \lambda(G)$.
\end{lemma}
\begin{proof}
Let $x$ be the Perron vector of $A(G')$, then $x(u)\leq x(v)$ or $x(u)\geq x(v)$. If $x(u)\leq x(v)$, then deleting edges $uu_2,...,uu_t$ and adding edges $vu_2,...,vu_t$ get the graph $G$, then $\lambda(G)\geq x^TA(G)x\geq x^TA(G')x=\lambda(G')$ with equality if and only if $x$ is the eigenvector of $A(G)$, but it is easy to find that $x$ is not the eigenvector of $A(G)$, so $\lambda(G')< \lambda(G)$; If $x(u)\geq x(v)$, then deleting edges $vu_1$ and adding edges $uu_1$ get the graph $G$, then $\lambda(G)\geq x^TA(G)x\geq x^TA(G')x=\lambda(G')$ with equality if and only if $x$ is the eigenvector of $A(G)$, but it is also easy to find that $x$ is not the eigenvector of $A(G)$, so $\lambda(G')< \lambda(G)$, this completes the proof.
\end{proof}

\begin{lemma}\label{L6}
Let $n=k\alpha>2\alpha$ and $G\in  \mathcal{G}_{n,\alpha}$  be  a graph obtained by joining an edge from a non-cut vertex of a graph $H\in \mathcal{G}_{n-k(l+p), \alpha-(l+p)} $ and a non-cut vertex of $P_{k(l+p),l+p}$, $H\neq K_k$ (see $Fig.3$). Let  $G'$ be the graph obtained from $G$ by deleting the edge $v_{pk}v_{p+1,1}$ and adding edge $v_{01}v_{p+1,1}$. Then $\lambda(G')> \lambda(G)$.\\ \\
\begin{picture}(60, 18)

\put(15,6){$K_k$}
\put(15,16){$V_{l+p}$}
\put(18,8){\oval(8,12)}
\put(22,8){\line(1,0){3}}
\put(22,8){\circle*{1}}
\put(26,8){$\ldots$}
\put(32,6){$K_k$}
\put(32,16){$V_{p+1}$}
\put(35,8){\oval(8,12)}
\put(39,8){\line(1,0){11}}
\put(39,8){\circle*{1}}
\put(50,8){\circle*{1}}
\put(39,5){$v_{p+1,1}$}
\put(44.5,10){$v_{pk}$}
\put(58.5,5){$v_{p1}$}
\put(61,10){$v_{p-1,k}$}
\put(79.5,5){$v_{p-1,1}$}
\put(51,6){$K_k$}
\put(51,16){$V_{p}$}
\put(54,8){\oval(8,12)}
\put(72,6){$K_k$}
\put(72,16){$V_{p-1}$}
\put(75,8){\oval(8,12)}
\put(58,8){\line(1,0){13}}
\put(79,8){\line(1,0){3}}
\put(58,8){\circle*{1}}
\put(71,8){\circle*{1}}
\put(79,8){\circle*{1}}
\put(84,8){$\ldots$}
\put(92,6){$K_k$}
\put(92,16){$V_1$}
\put(95,8){\oval(8,12)}
\put(99,8){\line(3,1){11}}
\put(99,8){\circle*{1}}
\put(110,11.5){\circle*{1}}
\put(91,8){\circle*{1}}
\put(110,5){\circle*{1}}
\put(107,14){$v_{0k}$}
\put(85.5,10){$v_{1k}$}
\put(99,5){$v_{11}$}
\put(107,1){$v_{01}$}
\put(111,6){$K_k$}
\put(111,16){$V_0$}
\put(106,22){$.............................$}
\put(106,-2.5){$\vdots$}
\put(106,2){$\vdots$}
\put(106,6){$\vdots$}
\put(106,10){$\vdots$}
\put(106,14){$\vdots$}
\put(106,18){$\vdots$}
\put(106,-2.5){$............................$}
\put(138,-2.5){$\vdots$}
\put(138,2){$\vdots$}
\put(138,6){$\vdots$}
\put(138,10){$\vdots$}
\put(138,14){$\vdots$}
\put(138,18){$\vdots$}

\put(114,8){\oval(8,12)}
\put(116,8){\line(1,0){7}}
\put(116,10){\line(1,0){7}}
\put(116,6){\line(1,0){7}}
\put(128,8){\circle{15}}
\put(128,18){$H$}
\put(139,6){$G$}
\end{picture}\\ \\
\begin{picture}(60,24)
\put(32,-4){$K_k$}
\put(35,-2){\oval(8,12)}
\put(39,-2){\line(1,0){11}}
\put(39,-2){\circle*{1}}
\put(50,-2){\circle*{1}}
\put(51,-4){$K_k$}
\put(32,6){$V_{l+p}$}
\put(51,6){$V_{l+p-1}$}
\put(54,-2){\oval(8,12)}
\put(58,-2){\line(1,0){4}}
\put(58,-2){\circle*{1}}
\put(64,-2){$\ldots$}
\put(73,-4){$K_k$}
\put(73,6){$V_{p+1}$}
\put(76,-2){\oval(8,12)}
\put(80,-2){\line(3,2){11}}
\put(80,-2){\circle*{1}}
\put(32,16){$K_k$}
\put(32,26){$V_{p}$}
\put(35,18){\oval(8,12)}
\put(51,16){$K_k$}
\put(51,26){$V_{p-1}$}
\put(54,18){\oval(8,12)}
\put(39,18){\line(1,0){11}}
\put(58,18){\line(1,0){4}}
\put(39,18){\circle*{1}}
\put(50,18){\circle*{1}}
\put(58,18){\circle*{1}}
\put(65,18){$\ldots$}
\put(73,15.5){$K_k$}
\put(73,26){$V_1$}
\put(76,18){\oval(8,12)}
\put(80,18){\line(3,-2){11}}
\put(80,18){\circle*{1}}
\put(72,18){\circle*{1}}
\put(31,18){\circle*{1}}
\put(91,5.2){\circle*{1}}
\put(91,10.7){\circle*{1}}
\put(92,6){$K_k$}
\put(92,16){$V_0$}
\put(25.5,20){$v_{pk}$}
\put(39,15){$v_{p1}$}
\put(80,-5){$v_{p+1,1}$}
\put(40,20){$v_{p-1,k}$}
\put(58,15){$v_{p-1,1}$}
\put(66.5,20){$v_{1k}$}
\put(80,20){$v_{11}$}
\put(87,14){$v_{0k}$}
\put(88,1){$v_{01}$}
\put(85,22){$.............................$}
\put(85,-2.5){$\vdots$}
\put(85,2){$\vdots$}
\put(85,6){$\vdots$}
\put(85,10){$\vdots$}
\put(85,14){$\vdots$}
\put(85,18){$\vdots$}
\put(85,-2.5){$............................$}
\put(117,-2.5){$\vdots$}
\put(117,2){$\vdots$}
\put(117,6){$\vdots$}
\put(117,10){$\vdots$}
\put(117,14){$\vdots$}
\put(117,18){$\vdots$}

\put(95,8){\oval(8,12)}

\put(97,10){\line(1,0){7}}
\put(97,6){\line(1,0){7}}
\put(109,8){\circle{15}}
\put(109,18){$H$}
\put(120,6){$G'$}
\put(69,-15){$Fig.3$}
\end{picture}\\ \\
\end{lemma}
\begin{proof}
Let $x$ be the Perron vector of $G$, by Theorem ~\ref{T1}, it is easy to get that the values of vertices with the same degree in the clique of order $k$ in $G-H$ are same. Let $V_{i}=\{v_{i1},v_{i2},...,v_{ik}\},i=0,1,...,l+p$. Thus $x(v_{i2})= x(v_{i3})= ... = x(v_{i,k-1}),i=1,2,...,p$, and without loss of generality, assume that $x(v_{02})\leq x(v_{03})\leq ... \leq x(v_{0,k-1})$. Next we will consider the following cases:\\
{\bf Case 1:} If $x(v_{01})\geq x(v_{pk})$,  then we delete the edge $v_{pk}v_{p+1,1}$ and add edge $v_{01}v_{p+1,1}$, then $\lambda(G)= x^TA(G)x= x^TA(G')x-2(x(v_{01})-x(v_{pk}))x_{p+1,1}\le \lambda(G')$ with equality holding if and only if $x$ is an eigenvector of $A(G')$. However, $(A(G')x)_{v_{pk}}=\sum_{i=1}^{k-1}x(v_{pi})<\lambda(G)x(v_{pk})$, so $x$ is not an eigenvector of $A(G')$. Thus $\lambda(G')>\lambda(G)$.\\
{\bf Case 2:} If $x(v_{01})< x(v_{pk})$, then we delete the edge $v_{pk}v_{p+1,1}$ and add edge $v_{01}v_{p+1,1}$. Next we will consider the following subcases:

 {\bf Subcase 2.1:} Suppose that $x(v_{02})\ge x(v_{p2})$. Let
 \begin{eqnarray*}
&& V_{01}=N_G(v_{01})\\
&&V_{0k}=N_G(v_{0k})\backslash \{v_{11},v_{01}\}\\
&&E_1= \{v_{pk}v_{p+1,1}\}\cup\{v_{01}w|w\in V_{01}\}\cup \{v_{0k}w|w\in V_{0k}\}\cup\{v_{pk}v_{pj}|j=1,2,..,k-1\}\\
&&E_2=\{uv\in E(G)| u,v \in \cup_{i=1}^{p-1}V_i\}\cup\{v_{0k}v_{11},v_{p-1,k}v_{p1}\}.
\end{eqnarray*}
For this case, it is divided into the following two subcases:

 {\bf Subcase 2.1.1:} If $x(v_{0k})\ge x(v_{p1})$, then let $y$ be a vector defined by the following:
\begin{displaymath}
y(u) = \left\{ \begin{array}{ll}
x(v_{01}), & \textrm{ $u=v_{pk}$}\\
x(v_{pk}),& \textrm{ $u=v_{01}$}\\
x(u),& \textrm{ otherwise}
\end{array} \right..
\end{displaymath}
Then
 \begin{eqnarray*}
&&y^TA(G')y-x^TA(G)x\\
&=&y^TA(G')y-y^TA(G)y+y^TA(G)y-x^TA(G)x\\
&=&2[y(v_{01})y(v_{p+1,1})-y(v_{pk})y(v_{p+1,1})]+y^TA(G)y-x^TA(G)x\\
&=&2[x(v_{pk})x(v_{p+1,1})-x(v_{01})x(v_{p+1,1})]+y^TA(G)y-x^TA(G)x.
\end{eqnarray*}
By the definition of $Y$, it implies that
\begin{displaymath}
y(u)y(v) = \left\{ \begin{array}{ll}
x(u)x(v), & \textrm{ $uv\notin E_1$}\\
x(v_{pk})x(v),& \textrm{ $u=v_{01},v\in V_{01}$}\\
x(u)x(v),& \textrm{ $u=v_{0k},v\in V_{0k}$}\\
x(v_{01})x(v),& \textrm{ $u=v_{pk},v=v_{pj},2\le j\le k-1$}\\
\end{array} \right.,
\end{displaymath}
so
 \begin{eqnarray*}
&&y^TA(G)y-x^TA(G)x\\
&=&2\sum_{uv\in E_1}[y(u)y(v)-x(u)x(v)]+2\sum_{uv\notin E_1}[y(u)y(v)-x(u)x(v)]\\
&=&2[y(v_{pk})y(v_{p+1,1})-x(v_{pk})x(v_{p+1,1})]+2\sum_{w\in V_{01}}[y(v_{01})y(w)-x(v_{01})x(w)]\\
&+&\sum_{w\in V_{0k}}2[y(v_{0k})y(w)-x(v_{0k})x(w)]+\sum_{j=1,...,k-1}2[y(v_{pk})y(v_{pj})-x(v_{pk})x(v_{pj})]\\
&=&2[x(v_{01})-x(v_{pk})]x(v_{p+1,1})+2\sum_{w\in V_{01}}[x(v_{pk})-x(v_{01})]x(w)\\
&-&2\sum_{j=1,...,k-1}[x(v_{pk})-x(v_{01})]x(v_{pj}),
\end{eqnarray*}
then
 \begin{eqnarray*}
&&y^TA(G')y-x^TA(G)x\\
&=&2[y(v_{01})y(v_{p+1,1})-y(v_{pk})y(v_{p+1,1})]+y^TA(G)y-x^TA(G)x\\
&=&2[x(v_{pk})x(v_{p+1,1})-x(v_{01})x(v_{p+1,1})]+2[x(v_{01})-x(v_{pk})]x(v_{p+1,1})\\
&+& 2\sum_{w\in V_{01}}[x(v_{pk})-x(v_{01})]x(w)-2\sum_{j=1,...,k-1}[x(v_{pk})-x(v_{01})]x(v_{pj})\\
&=&2\sum_{w\in V_{01}}[x(v_{pk})-x(v_{01})]x(w)-2\sum_{j=1,...,k-1}[x(v_{pk})-x(v_{01})]x(v_{pj})\\
&\ge&2\sum_{j=2,...,k-1}[x(v_{pk})-x(v_{01})][x(v_{0j})-x(v_{pj})]+[x(v_{pk})-x(v_{01})][x(v_{0k})-x(v_{p1})].
\end{eqnarray*}
By $x(v_{pk})>x(v_{01}),x(v_{0,k-1})\ge ... \ge x(v_{02})\ge x(v_{p2})=...=x(v_{p,k-1})$ and $x(v_{0k})\ge x(v_{p1})$, so $\lambda(G')\ge y^TA(G')y\ge x^TA(G)x=\lambda(G)$ with equality if and only if $y$ is the Perron vector of $A(G')$. However,
\begin{eqnarray*}
(A(G')y)_{v_{p2}}&=&\sum_{j=1,3,4,..,k}y(v_{pj})=y(v_{pk})+\sum_{j=1,3,4,..,k-1}y(v_{pj})\\
&=&x(v_{01})+\sum_{j=1,3,4,..,k-1}x(v_{pj})\\
&<&x(v_{pk})+\sum_{j=1,3,4,..,k-1}x(v_{pj})\\
&=&\lambda(G)x(v_{p2})=\lambda(G)y(v_{p2}).
\end{eqnarray*}
So $y$ is not the Perron vector of $A(G')$, thus $\lambda(G')>\lambda(G)$.

{\bf Subcase 2.1.2:} If $x(v_{0k})< x(v_{p1})$, then let $y$ be a vector defined by the following:
\begin{displaymath}
y(u) = \left\{ \begin{array}{ll}
x(v_{p-i,k+1-j}), & \textrm{ $u=v_{ij},i\in \{0,1,...,p\},j\in\{1,k\}$}\\
x(v_{p-i,j}), & \textrm{ $u=v_{ij},i=1,2,...,p-1;j=2,3,...,k-1$}\\
x(w),& \textrm{ otherwise}
\end{array} \right..
\end{displaymath}
Then
 \begin{eqnarray*}
&&y^TA(G')y-x^TA(G)x\\
&=&y^TA(G')y-y^TA(G)y+y^TA(G)y-x^TA(G)x\\
&=&2[y(v_{01})y(v_{p+1,1})-y(v_{pk})y(v_{p+1,1})]+y^TA(G)y-x^TA(G)x\\
&=&2[x(v_{pk})-x(v_{01})]x(v_{p+1,1})+y^TA(G)y-x^TA(G)x.
\end{eqnarray*}
By the definition of $y$, it implies that
\begin{displaymath}
y(u)y(v) = \left\{ \begin{array}{ll}
x(u)x(v), & \textrm{ $uv\notin E_1\cup E_2\cup\{v_{p1}v_{pj}:j=2,...,k-1\}$}\\
x(v_{pk})x(v),& \textrm{ $u=v_{01},v\in V_{01}$}\\
x(v_{p1})x(v),& \textrm{ $u=v_{0k},v\in V_{0k}$}\\
x(v_{01})x(v),& \textrm{ $u=v_{pk},v=v_{pj},1\le j\le k-1$}
\end{array} \right.,
\end{displaymath}
so
 \begin{eqnarray*}
&&y^TA(G)y-x^TA(G)x\\
&=&2\sum_{uv\in E_1 \cup E_2}[y(u)y(v)-x(u)x(v)]+2\sum_{uv\notin E_1 \cup E_2}[y(u)y(v)-x(u)x(v)]\\
&=&2[y(v_{pk})y(v_{p+1,1})-x(v_{pk})x(v_{p+1,1})]+2\sum_{w\in V_{01}}[y(v_{01})y(w)-x(v_{01})x(w)]\\
&+&\sum_{w\in V_{0k}}2[y(v_{0k})y(w)-x(v_{0k})x(w)]+\sum_{j=1,...,k-1}2[y(v_{pk})y(v_{pj})-x(v_{pk})x(v_{pj})]\\
&+&2\sum_{uv\in E_2}[y(u)y(v)-x(u)x(v)]+2\sum_{j=2}^{k-1}[y(v_{p1})y(v_{pj})-x(v_{p1})x(v_{pj})]\\
&=&2[x(v_{01})x(v_{p+1,1})-x(v_{pk})x(v_{p+1,1})]+2\sum_{w\in V_{01}}[x(v_{pk})y(w)-x(v_{01})x(w)]\\
&+&\sum_{w\in V_{0k}}2[x(v_{p1})x(w)-x(v_{0k})x(w)]+\sum_{j=1,...,k-1}2[x(v_{01})y(v_{pj})-x(v_{pk})x(v_{pj})]\\
&+&2\sum_{uv\in E_2}[y(u)y(v)-x(u)x(v)]+2\sum_{j=2}^{k-1}[x(v_{0k})x(v_{pj})-x(v_{p1})x(v_{pj})]\\
&=&2[x(v_{01})-x(v_{pk})]x(v_{p+1,1})+2\sum_{w\in V_{01}\backslash\{v_{0k}\}}[x(v_{pk})y(w)-x(v_{01})x(w)]\\
&+&\sum_{w\in V_{0k}}2[x(v_{p1})x(w)-x(v_{0k})x(w)]+\sum_{j=2,...,k-1}2[x(v_{01})y(v_{pj})-x(v_{pk})x(v_{pj})]\\
&+&2\sum_{uv\in E_2}[y(u)y(v)-x(u)x(v)]+2[x(v_{pk})y(v_{0k})-x(v_{01})x(v_{0k})]\\
&+&2[x(v_{01})y(v_{p1})-x(v_{pk})x(v_{p1})]+2\sum_{j=2}^{k-1}[x(v_{0k})-x(v_{p1})]x(v_{pj})\\
&=&2[x(v_{01})-x(v_{pk})]x(v_{p+1,1})+2\sum_{w\in V_{01}\backslash\{v_{0k}\}}[x(v_{pk})x(w)-x(v_{01})x(w)]\\
&+&\sum_{w\in V_{0k}}2[x(v_{p1})x(w)-x(v_{0k})x(w)]+\sum_{j=2,...,k-1}2[x(v_{01})x(v_{pj})-x(v_{pk})x(v_{pj})]\\
&+&2\sum_{uv\in E_2}[y(u)y(v)-x(u)x(v)]+2[x(v_{pk})x(v_{p1})-x(v_{01})x(v_{0k})]\\
&+&2[x(v_{01})x(v_{0k})-x(v_{pk})x(v_{p1})]+2\sum_{j=2}^{k-1}[x(v_{0k})-x(v_{p1})]x(v_{pj})\\
&=&2[x(v_{01})-x(v_{pk})]x(v_{p+1,1})+2\sum_{w\in V_{01}\backslash\{v_{0k}\}}[x(v_{pk})-x(v_{01})]x(w)\\
&+&\sum_{w\in V_{0k}}2[x(v_{p1})-x(v_{0k})]x(w)+\sum_{j=2,...,k-1}2[x(v_{01})-x(v_{pk})]x(v_{pj})\\
&+&2\sum_{uv\in E_2}[y(u)y(v)-x(u)x(v)]+2\sum_{j=2}^{k-1}[x(v_{0k})-x(v_{p1})]x(v_{pj}),
\end{eqnarray*}
By the definition of $y$, it also implies that
\begin{eqnarray*}
&&\sum_{i=1,...,p-1}\sum_{1=j_1< j_2<k}\left\{[y(v_{ij_1})y(v_{ij_2})-x(v_{ij_1})x(v_{ij_2})]+[y(v_{p-i,j_1})y(v_{p-i,j_2})-x(v_{p-i,j_1})x(v_{p-i,j_2})]\right\}\\
&+&\sum_{i=1,...,p-1}\sum_{1<j_1< j_2=k}\left\{[y(v_{ij_1})y(v_{ij_2})-x(v_{ij_1})x(v_{ij_2})]+[y(v_{p-i,j_1})y(v_{p-i,j_2})-x(v_{p-i,j_1})x(v_{p-i,j_2})]\right\}\\
&=&\sum_{i=1,...,p-1}\sum_{1<j_1< j_2<k}\left\{[y(v_{ij_1})y(v_{ij_2})-x(v_{ij_1})x(v_{ij_2})]+[y(v_{p-i,j_1})y(v_{p-i,j_2})-x(v_{p-i,j_1})x(v_{p-i,j_2})]\right\}\\
&=&\sum_{i=1,...,p-1}\sum_{j_1=1, j_2=k}\left\{[y(v_{ij_1})y(v_{ij_2})-x(v_{ij_1})x(v_{ij_2})]+[y(v_{p-i,j_1})y(v_{p-i,j_2})-x(v_{p-i,j_1})x(v_{p-i,j_2})]\right\}\\
&=&0.
\end{eqnarray*}
So
\begin{eqnarray*}
&&2\sum_{uv\in E_2}[y(u)y(v)-x(u)x(v)]\\
&=&\sum_{i=0,...,p-1}\{[y(v_{ik})y(v_{i+1,1})-x(v_{ik})x(v_{i+1,1})]+[y(v_{p-i-1,k})y(v_{p-i,1})-x(v_{p-i-1,k})x(v_{p-i,1})]\}\\
&+&\sum_{i=1,...,p-1}\sum_{j_1< j_2}\left\{[y(v_{ij_1})y(v_{ij_2})-x(v_{ij_1})x(v_{ij_2})]+[y(v_{p-i,j_1})y(v_{p-i,j_2})-x(v_{p-i,j_1})x(v_{p-i,j_2})]\right\}\\
&=&\sum_{i=0,...,p-1}\{[x(v_{p-i,1})x(v_{p-i-1,k})-x(v_{ik})x(v_{i+1,1})]+[x(v_{i+1,1})x(v_{ik})-x(v_{p-i-1,k})x(v_{p-i,1})]\}\\
&+&\sum_{i=1,...,p-1}\sum_{1=j_1< j_2<k}\left\{[y(v_{ij_1})y(v_{ij_2})-x(v_{ij_1})x(v_{ij_2})]+[y(v_{p-i,j_1})y(v_{p-i,j_2})-x(v_{p-i,j_1})x(v_{p-i,j_2})]\right\}\\
&+&\sum_{i=1,...,p-1}\sum_{1<j_1< j_2<k}\left\{[y(v_{ij_1})y(v_{ij_2})-x(v_{ij_1})x(v_{ij_2})]+[y(v_{p-i,j_1})y(v_{p-i,j_2})-x(v_{p-i,j_1})x(v_{p-i,j_2})]\right\}\\
&+&\sum_{i=1,...,p-1}\sum_{1<j_1< j_2=k}\left\{[y(v_{ij_1})y(v_{ij_2})-x(v_{ij_1})x(v_{ij_2})]+[y(v_{p-i,j_1})y(v_{p-i,j_2})-x(v_{p-i,j_1})x(v_{p-i,j_2})]\right\}\\
&+&\sum_{i=1,...,p-1}\sum_{j_1=1, j_2=k}\left\{[y(v_{ij_1})y(v_{ij_2})-x(v_{ij_1})x(v_{ij_2})]+[y(v_{p-i,j_1})y(v_{p-i,j_2})-x(v_{p-i,j_1})x(v_{p-i,j_2})]\right\}
\\&=&0.
\end{eqnarray*}

Then
 \begin{eqnarray*}
&&y^TA(G)y-x^TA(G)x\\
&=&2[x(v_{01})-x(v_{pk})]x(v_{p+1,1})+2\sum_{w\in V_{01}\backslash\{v_{0k}\}}[x(v_{pk})-x(v_{01})]x(w)\\
&+&2\sum_{w\in V_{0k}}[x(v_{p1})-x(v_{0k})]x(w)+2\sum_{j=2,...,k-1}[x(v_{01})-x(v_{pk})]x(v_{pj})\\
&+&2\sum_{j=2}^{k-1}[x(v_{0k})-x(v_{p1})]x(v_{pj}),
\end{eqnarray*}
so
\begin{eqnarray*}
&&y^TA(G')y-x^TA(G)x\\
&=&2[x(v_{pk})-x(v_{01})]x(v_{p+1,1})+2[x(v_{01})-x(v_{pk})]x(v_{p+1,1})\\
&+&\sum_{w\in V_{01}\backslash\{v_{0k}\}}2[x(v_{pk})-x(v_{01})]x(w)+\sum_{w\in V_{0k}}2[x(v_{p1})-x(v_{0k})]x(w)\\
&+&\sum_{j=2,...,k-1}2[x(v_{01})-x(v_{pk})]x(v_{pj})+2\sum_{j=2}^{k-1}[x(v_{0k})-x(v_{p1})]x(v_{pj})\\
&=&\sum_{w\in V_{01}\backslash\{v_{0k}\}}2[x(v_{pk})-x(v_{01})]x(w)+\sum_{w\in V_{0k}}2[x(v_{p1})-x(v_{0k})]x(w)\\
&+&\sum_{j=2,...,k-1}2[x(v_{01})-x(v_{pk})]x(v_{pj})+2\sum_{j=2}^{k-1}[x(v_{0k})-x(v_{p1})]x(v_{pj})\\
&\ge&2\sum_{j=2}^{k-1}[x(v_{p1})-x(v_{0k})][x(v_{0j})-x(v_{pj})]+2\sum_{j=2,...,k-1}[x(v_{pk})-x(v_{01})][x(v_{0j}-x(v_{pj})].
\end{eqnarray*}
Since $x(v_{p1})>x(v_{0k}),x(v_{pk})>x(v_{01})$  and $x(v_{02})\ge x(v_{p2})$, so $\lambda(G')\ge y^TA(G')y\ge x^TA(G)x=\lambda(G)$ with equality if and only if $y$ is the Perron vector of $A(G')$. However,
\begin{eqnarray*}
(A(G')y)_{v_{p2}}&=&\sum_{j=1,3,4,..,k}y(v_{pj})=y(v_{p1})+y(v_{pk})+\sum_{j=3,4,..,k-1}y(v_{pj})\\
&=&x(v_{0k})+x(v_{01})+\sum_{j=3,4,..,k-1}x(v_{pj})\\
&<&x(v_{p1})+x(v_{pk})+\sum_{j=3,4,..,k-1}x(v_{pj})\\
&=&\lambda(G)x(v_{p2})=\lambda(G)y(v_{p2}).
\end{eqnarray*}
So $y$ is not the Perron vector of $A(G')$, thus $\lambda(G')>\lambda(G)$.

{\bf Subcase 2.2:} If there exists a integer $t$ satisfies that $2<t<k-1$ and $x(v_{0t})<x(v_{p2})\le x(v_{0,t+1})$. Let
 \begin{eqnarray*}
V_{01}'&=&\left(\bigcup_{i=1,...,t}N_G(v_{0i})\right) \Big{\backslash}\Big{\{}v_{0i}\Big{\}}_{i=1}^t\\
V_{0k}'&=&\{w|w\sim v_{0k},w\in V(G),w\notin \{v_{11},v_{01},...,v_{0t}\}\}=N_G(v_{0k})\backslash \{v_{11},v_{01},...,v_{0t}\}\\
E_1'&= &\{v_{pk}v_{p+1,1}\}\cup\{v_{0j}w\in E(G)|w\in V_{01}',j=1,..,t\}\cup \{v_{0k}w|w\in V_{0k}'\}\cup\\
&&\{v_{0i}v_{0j}|1\le i<j\le t\}\cup\{v_{pi}v_{pj}|i\in \{2,3,..,t,k\},i\neq j\}\\
E_2'&=&\{uv\in E(G)| u,v \in \cup_{i=1}^{k-1}V_i\}\cup\{v_{0k}v_{11},v_{p-1,k}v_{p1}\}=E_2.
\end{eqnarray*}
Then it is divided into the following two cases.\\
{\bf Subcase 2.2.1:} If $x(v_{0k})\ge x(v_{p1})$, then let $y$ be a vector defined by the following:
\begin{displaymath}
y(u) = \left\{ \begin{array}{ll}
x(v_{01}), & \textrm{ $u=v_{pk}$}\\
x(v_{pk}),& \textrm{ $u=v_{01}$}\\
x(v_{p-i,j}),& \textrm{ $u=v_{ij},i\in \{0,p\},j=2,...,t$}\\
x(u),& \textrm{ otherwise}
\end{array} \right..
\end{displaymath}
Then
 \begin{eqnarray*}
&&y^TA(G')y-x^TA(G)x\\
&=&y^TA(G')y-y^TA(G)y+y^TA(G)y-x^TA(G)x\\
&=&2[y(v_{01})y(v_{p+1,1})-y(v_{pk})y(v_{p+1,1})]+y^TA(G)y-x^TA(G)x\\
&=&2[x(v_{pk})x(v_{p+1,1})-x(v_{01})x(v_{p+1,1})]+y^TA(G)y-x^TA(G)x.
\end{eqnarray*}
By the definition of $ y $, it implies that $y(u)y(v) =x(u)x(v)$, if $uv\notin E_1'$. Then
\begin{eqnarray*}
&&y^TA(G)y-x^TA(G)x\\
&=&2\sum_{uv\in E_1'}[y(u)y(v)-x(u)x(v)]+2\sum_{uv\notin E_1'}[y(u)y(v)-x(u)x(v)]\\
&=&2[y(v_{pk})y(v_{p+1,1})-x(v_{pk})x(v_{p+1,1})]+2\sum_{j=1}^t\sum_{w\sim v_{0j},w\in V_{01}'}[y(v_{0j})y(w)-x(v_{0j})x(w)]\\
&+&\sum_{w\in V_{0k}'}2[y(v_{0k})y(w)-x(v_{0k})x(w)]+\sum_{1\le i<j\le t}2[y(v_{0i})y(v_{0j})-x(v_{0i})x(v_{0j})]\\
&+&\sum_{i\in \{2,3,..,t,k\}, j\notin \{2,3,..,t,k\}}2[y(v_{pi})y(v_{pj})-x(v_{pi})x(v_{pj})]\\
&+&\sum_{i,j\in \{2,3,..,t,k\},i< j}2[y(v_{pi})y(v_{pj})-x(v_{pi})x(v_{pj})]\\
&=&2[x(v_{01})-x(v_{pk})]x(v_{p+1,1})+2\sum_{j=2}^t\sum_{w\sim v_{0j},w\in V_{01}'}[y(v_{0j})y(w)-x(v_{0j})x(w)]\\
&+&\sum_{w\in V_{0k}'}2[x(v_{0k})x(w)-x(v_{0k})x(w)]+\sum_{2\le i<j\le t}2[x(v_{pi})x(v_{pj})-x(v_{0i})x(v_{0j})]\\
&+&\sum_{i=2}^t\sum_{j\in \{1,t+1,t+2,...,k-1\}}2[y(v_{pi})y(v_{pj})-x(v_{pi})x(v_{pj})]\\
&+&\sum_{2\le i< j\le t}2[x(v_{0i})x(v_{0j})-x(v_{pi})x(v_{pj})]+\sum_{w\in N_G(v_{01})}2[y(v_{01})y(w)-x(v_{01})x(w)]\\
&+&\sum_{w\in N_G(v_{pk})\backslash\{v_{p+1,1}\}}2[y(v_{pk})y(w)-x(v_{pk})x(w)]\\
&=&2[x(v_{01})-x(v_{pk})]x(v_{p+1,1})+2\sum_{i=2}^t\sum_{w\sim v_{0i},w\in V_{01}'}[x(v_{pi})x(w)-x(v_{0i})x(w)]\\
&+&\sum_{i=2}^t\sum_{j\in \{1,t+1,t+2,...,k-1\}}2[x(v_{0i})x(v_{pj})-x(v_{pi})x(v_{pj})]\\
&+&\sum_{i=2}^t2[y(v_{01})y(v_{0i})-x(v_{01})x(v_{0i})]+\sum_{i=2}^t2[y(v_{pk})y(v_{pi})-x(v_{pk})x(v_{pi})]\\
&+&\sum_{w\in N_G(v_{01})\backslash \{v_{02},...,v_{0t}\}}2[y(v_{01})y(w)-x(v_{01})x(w)]\\
&+&\sum_{w\in N_G(v_{pk})\backslash\{v_{p+1,1},v_{p2},...,v_{pt}\}}2[y(v_{pk})y(w)-x(v_{pk})x(w)]\\
&=&2[x(v_{01})-x(v_{pk})]x(v_{p+1,1})+2\sum_{i=2}^t\sum_{w\sim v_{0i},w\in V_{01}'}[x(v_{pi})-x(v_{0i})]x(w)\\
&+&\sum_{i=2}^t\sum_{j\in \{1,t+1,t+2,...,k-1\}}2[x(v_{0i})-x(v_{pi})]x(v_{pj})\\
&+&\sum_{i=2}^t2[x(v_{pk})x(v_{pi})-x(v_{01})x(v_{0i})]+\sum_{i=2}^t2[x(v_{01})x(v_{0i})-x(v_{pk})x(v_{pi})]\\
&+&\sum_{w\in N_G(v_{01})\backslash \{v_{02},...,v_{0t}\}}2[x(v_{pk})x(w)-x(v_{01})x(w)]\\
&+&\sum_{w\in N_G(v_{pk})\backslash\{v_{p+1,1},v_{p2},...,v_{pt}\}}2[x(v_{01})x(w)-x(v_{pk})x(w)]\\
&=&2[x(v_{01})-x(v_{pk})]x(v_{p+1,1})+2\sum_{i=2}^t\sum_{w\sim v_{0i},w\in V_{01}'}[x(v_{pi})-x(v_{0i})]x(w)\\
&+&\sum_{i=2}^t\sum_{j\in \{1,t+1,t+2,...,k-1\}}2[x(v_{0i})-x(v_{pi})]x(v_{pj})\\
&+&\sum_{w\in N_G(v_{01})\backslash \{v_{02},...,v_{0t},v_{0k}\}}2[x(v_{pk})-x(v_{01})]x(w)+2[x(v_{pk})-x(v_{01})]x(v_{0k})\\
&+&\sum_{w\in N_G(v_{pk})\backslash\{v_{p+1,1},v_{p1},v_{p2},...,v_{pt}\}}2[x(v_{01})-x(v_{pk})]x(w)+2[x(v_{01})-x(v_{pk})]x(v_{p1})\\
\end{eqnarray*}
Since $x(v_{p1})\le x(v_{0k}),x(v_{pk})>x(v_{01})$, $x(v_{0t})< x(v_{p2})\le x(v_{0,t+1})$ and $H\neq K_k$, then
\begin{eqnarray*}
&&y^TA(G)y-x^TA(G)x\\
&\ge&2[x(v_{01})-x(v_{pk})]x(v_{p+1,1})+2\sum_{i=2}^t\sum_{w\sim v_{0j},w\in V_{01}'}[x(v_{pi})-x(v_{0i})]x(w)\\
&+&\sum_{i=2}^t\sum_{j\in \{1,t+1,t+2,...,k-1\}}2[x(v_{0i})-x(v_{pi})]x(v_{pj})\\
&+&\sum_{i=t+1}^{k-1}2[x(v_{pk})-x(v_{01})][x(v_{0i})-x(v_{pi})]+2[x(v_{pk})-x(v_{01})][x(v_{0k})-x(v_{p1})]\\
&\ge&2[x(v_{01})-x(v_{pk})]x(v_{p+1,1})+\sum_{i=2}^t\sum_{j=t+1}^{k-1}2[x(v_{pi})-x(v_{0i})][x(v_{0j})-x(v_{pj})]\\
&+&\sum_{i=2}^t2[x(v_{pi})-x(v_{0i})]x(v_{0k})+\sum_{i=2}^t2[x(v_{0i})-x(v_{pi})]x(v_{p1})\\
&+&\sum_{i=t+1}^{k-1}2[x(v_{pk})-x(v_{01})][x(v_{0i})-x(v_{pi})]+2[x(v_{pk})-x(v_{01})][x(v_{0k})-x(v_{p1})]\\
&=&2[x(v_{01})-x(v_{pk})]x(v_{p+1,1})+\sum_{i=2}^t\sum_{j=t+1}^{k-1}2[x(v_{pi})-x(v_{0i})][x(v_{0j})-x(v_{pj})]\\
&+&\sum_{i=2}^t2[x(v_{pi})-x(v_{0i})][x(v_{0k})-x(v_{p1})]+\sum_{i=t+1}^{k-1}2[x(v_{pk})-x(v_{01})][x(v_{0i}-x(v_{pi})]\\
&+&2[x(v_{pk})-x(v_{01})][x(v_{0k})-x(v_{p1})],
\end{eqnarray*}
thus
 \begin{eqnarray*}
&&y^TA(G')y-x^TA(G)x\\
&=&2[y(v_{01})y(v_{p+1,1})-y(v_{pk})y(v_{p+1,1})]+y^TA(G)y-x^TA(G)x\\
&\ge&2[x(v_{pk})x(v_{p+1,1})-x(v_{01})x(v_{p+1,1})]+2[x(v_{01})-x(v_{pk})]x(v_{p+1,1})\\
&+&\sum_{i=2}^t\sum_{j=t+1}^{k-1}2[x(v_{pi})-x(v_{0i})][x(v_{0j})-x(v_{pj})]+\sum_{i=2}^t2[x(v_{pi})-x(v_{0i})][x(v_{0k})-x(v_{p1})]\\
&+&\sum_{i=t+1}^{k-1}2[x(v_{pk})-x(v_{01})][x(v_{0i})-x(v_{pi})]+2[x(v_{pk})-x(v_{01})][x(v_{0k})-x(v_{p1})]\\
&=&\sum_{i=2}^t\sum_{j=t+1}^{k-1}2[x(v_{pi})-x(v_{0i})][x(v_{0j})-x(v_{pj})]+\sum_{i=2}^t2[x(v_{pi})-x(v_{0i})][x(v_{0k})-x(v_{p1})]\\
&+&\sum_{i=t+1}^{k-1}2[x(v_{pk})-x(v_{01})][x(v_{0i})-x(v_{pi})]+2[x(v_{pk})-x(v_{01})][x(v_{0k})-x(v_{p1})]\\
&\ge& 0.
\end{eqnarray*}
So $\lambda(G')\ge y^TA(G')y\ge x^TA(G)x=\lambda(G)$ with equality if and only if $y$ is the Perron vector of $A(G')$. However,
\begin{eqnarray*}
(A(G')y)_{v_{p,k-1}}&=&\sum_{j=1,2,..,k-2,k}y(v_{pj})=y(v_{pk})+y(v_{p1})+\sum_{j=2,..,t}y(v_{pj})+\sum_{j=t+1,..,k-2}y(v_{pj})\\
&=&x(v_{01})+x(v_{p1})+\sum_{j=2,..,t}x(v_{0j})+\sum_{j=t+1,..,k-2}x(v_{pj})\\
&<&x(v_{pk})+x(v_{p1})+\sum_{j=2,..,t}x(v_{pj})+\sum_{j=t+1,..,k-2}x(v_{pj})\\
&=&\lambda(G)x(v_{p,k-1})=\lambda(G)y(v_{p,k-1}).
\end{eqnarray*}
So $y$ is not the Perron vector of $A(G')$, thus $\lambda(G')>\lambda(G)$.

{\bf Subcase 2.2.2:} If $x(v_{0k})< x(v_{p1})$, then let $y$ be a vector defined by the following:
\begin{displaymath}
y(u) = \left\{ \begin{array}{ll}
x(v_{p-i,k+1-j}), & \textrm{ $u=v_{ij},i\in \{0,1,...,p\},j\in\{1,k\}$}\\
x(v_{p-i,j}), & \textrm{ $u=v_{ij},i=1,2,...,p-1;j=2,3,...,k-1$}\\
x(v_{p-i,j}),& \textrm{ $u=v_{ij},i\in \{0,p\},j=2,...,t$}\\
x(w),& \textrm{ otherwise}
\end{array} \right..
\end{displaymath}
Then
 \begin{eqnarray*}
&&y^TA(G')y-x^TA(G)x\\
&=&y^TA(G')y-y^TA(G)y+y^TA(G)y-x^TA(G)x\\
&=&2[y(v_{01})y(v_{p+1,1})-y(v_{pk})y(v_{p+1,1})]+y^TA(G)y-x^TA(G)x\\
&=&2[x(v_{pk})-x(v_{01})]x(v_{p+1,1})+y^TA(G)y-x^TA(G)x.
\end{eqnarray*}
For $uv\notin E_1'\cup E_2'\cup\{v_{p1}v_{pj},t<j<k\}$, $y(u)y(v)=x(u)x(v)$, by similar calculation in the {\bf Subcase 2.1.2}, then we can find that $2\sum_{uv\in E_2'}[y(u)y(v)-x(u)x(v)]=0$. Thus
 \begin{eqnarray*}
&&y^TA(G)y-x^TA(G)x\\
&=&2\sum_{uv\in E_1' \cup E_2'}[y(u)y(v)-x(u)x(v)]+2\sum_{uv\notin E_1' \cup E_2'}[y(u)y(v)-x(u)x(v)]\\
&=&2[y(v_{pk})y(v_{p+1,1})-x(v_{pk})x(v_{p+1,1})]+2\sum_{j=1}^t\sum_{w\sim v_{0j},w\in V_{01}'}[y(v_{0j})y(w)-x(v_{0j})x(w)]\\
&+&2\sum_{w\in V_{0k}'}[y(v_{0k})y(w)-x(v_{0k})x(w)]+2\sum_{1\le i<j\le t}[y(v_{0i})y(v_{0j})-x(v_{0i})x(v_{0j})]\\
&+&2\sum_{i\in \{2,3,..,t,k\}, j\notin \{2,3,..,t,k\}}[y(v_{pi})y(v_{pj})-x(v_{pi})x(v_{pj})]\\
&+&2\sum_{i,j\in \{2,3,..,t,k\},i< j}[y(v_{pi})y(v_{pj})-x(v_{pi})x(v_{pj})]+\sum_{j=t+1}^{k-1}[y(v_{p1})y(v_{pj})-x(v_{p1})x(v_{pj})]\\
&=&2[x(v_{01})-x(v_{pk})]x(v_{p+1,1})+2\sum_{j=2}^t\sum_{w\sim v_{0j},w\in V_{01}'}[y(v_{0j})y(w)-x(v_{0j})x(w)]\\
&+&2\sum_{w\in V_{0k}'}[x(v_{p1})x(w)-x(v_{0k})x(w)]+2\sum_{2\le i<j\le t}[x(v_{pi})x(v_{pj})-x(v_{0i})x(v_{0j})]\\
&+&2\sum_{i=2}^t\sum_{j\in \{1,t+1,t+2,...,k-1\}}[y(v_{pi})y(v_{pj})-x(v_{pi})x(v_{pj})]\\
&+&2\sum_{2\le i< j\le t}[x(v_{0i})x(v_{0j})-x(v_{pi})x(v_{pj})]+2\sum_{w\in N_G(v_{01})}[y(v_{01})y(w)-x(v_{01})x(w)]\\
&+&2\sum_{w\in N_G(v_{pk})\backslash\{v_{p+1,1}\}}[y(v_{pk})y(w)-x(v_{pk})x(w)]+\sum_{j=t+1}^{k-1}[x(v_{0k})x(v_{pj})-x(v_{p1})x(v_{pj})]\\
&=&2[x(v_{01})-x(v_{pk})]x(v_{p+1,1})+2\sum_{j=2}^t\sum_{w\sim v_{0j},w\in V_{01}'\backslash \{v_{0k}\}}[y(v_{0j})y(w)-x(v_{0j})x(w)]\\
&+&2\sum_{j=2}^t[y(v_{0j})y(v_{0k})-x(v_{0j})x(v_{0k})]+2\sum_{w\in V_{0k}'}[x(v_{p1})x(w)-x(v_{0k})x(w)]\\
&+&2\sum_{i=2}^t\sum_{j=t+1}^{k-1}[y(v_{pi})y(v_{pj})-x(v_{pi})x(v_{pj})]+\sum_{i=2}^t[y(v_{pi})y(v_{p1})-x(v_{pi})x(v_{p1})]\\
&+&2\sum_{w\in N_G(v_{01})}[y(v_{01})y(w)-x(v_{01})x(w)]+2\sum_{w\in N_G(v_{pk})\backslash\{v_{p+1,1}\}}[y(v_{pk})y(w)-x(v_{pk})x(w)]\\
&+&\sum_{j=t+1}^{k-1}[x(v_{0k})-x(v_{p1})]x(v_{pj})\\
&=&2[x(v_{01})-x(v_{pk})]x(v_{p+1,1})+2\sum_{j=2}^t\sum_{w\sim v_{0j},w\in V_{01}'\backslash \{v_{0k}\}}[x(v_{pj})x(w)-x(v_{0j})x(w)]\\
&+&2\sum_{j=2}^t[x(v_{pj})x(v_{p1})-x(v_{0j})x(v_{0k})]+2\sum_{w\in V_{0k}'}[x(v_{p1})x(w)-x(v_{0k})x(w)]\\
&+&2\sum_{i=2}^t\sum_{j=t+1}^{k-1}[x(v_{0i})x(v_{pj})-x(v_{pi})x(v_{pj})]+\sum_{i=2}^t[x(v_{0i})x(v_{0k})-x(v_{pi})x(v_{p1})]\\
&+&2\sum_{w\in N_G(v_{01})}[y(v_{01})y(w)-x(v_{01})x(w)]+2\sum_{w\in N_G(v_{pk})\backslash\{v_{p+1,1}\}}[y(v_{pk})y(w)-x(v_{pk})x(w)]\\
&+&\sum_{j=t+1}^{k-1}[x(v_{0k})-x(v_{p1})]x(v_{pj}).
\end{eqnarray*}
By $x(v_{p1})> x(v_{0k}),x(v_{pk})>x(v_{01})$  and $x(v_{0t})< x(v_{p2})\le x(v_{0,t+1})$, then
 \begin{eqnarray*}
&&y^TA(G)y-x^TA(G)x\\
&\ge&2[x(v_{01})-x(v_{pk})]x(v_{p+1,1})+2\sum_{j=t+1}^{k-1}[x(v_{p1})-x(v_{0k})][x(v_{0j})-x(v_{pj})]\\
&+&2\sum_{i=2}^t\sum_{j=t+1}^{k-1}[x(v_{pi})-x(v_{0i})][x(v_{0j})-x(v_{pj})]\\
&+&2\sum_{i=2}^t[y(v_{01})y(v_{0i})-x(v_{01})x(v_{0i})]+[y(v_{01})y(v_{0k})-x(v_{01})x(v_{0k})]\\
&+&2\sum_{w\in N_G(v_{01})\backslash\{v_{02},...,v_{0t},v_{0k}\}}[y(v_{01})y(w)-x(v_{01})x(w)]\\
&+&2\sum_{i=t+1}^{k-1}[y(v_{pk})y(v_{pi})-x(v_{pk})x(v_{pi})]+2\sum_{i=2}^t[y(v_{pk})y(v_{pi})-x(v_{pk})x(v_{pi})]\\
&+&[y(v_{pk})y(v_{p1})-x(v_{pk})x(v_{p1})]\\
&=&2[x(v_{01})-x(v_{pk})]x(v_{p+1,1})+2\sum_{j=t+1}^{k-1}[x(v_{p1})-x(v_{0k})][x(v_{0j})-x(v_{pj})]\\
&+&2\sum_{i=2}^t\sum_{j=t+1}^{k-1}[x(v_{pi})-x(v_{0i})][x(v_{0j})-x(v_{pj})]\\
&+&2\sum_{i=2}^t[x(v_{pk})x(v_{pi})-x(v_{01})x(v_{0i})]+[x(v_{pk})x(v_{p1})-x(v_{01})x(v_{0k})]\\\\
&+&2\sum_{w\in N_G(v_{01})\backslash\{v_{02},...,v_{0t},v_{0k}\}}[x(v_{pk})x(w)-x(v_{01})x(w)]\\
&+&2\sum_{i=t+1}^{k-1}[x(v_{01})x(v_{pi})-x(v_{pk})x(v_{pi})]+2\sum_{i=2}^t[x(v_{01})x(v_{0i})-x(v_{pk})x(v_{pi})]\\
&+&2[x(v_{01})x(v_{0k})-x(v_{pk})x(v_{p1})]\\
&\ge&2[x(v_{01})-x(v_{pk})]x(v_{p+1,1})+2\sum_{j=t+1}^{k-1}[x(v_{p1})-x(v_{0k})][x(v_{0j})-x(v_{pj})]\\
&+&2\sum_{i=2}^t\sum_{j=t+1}^{k-1}[x(v_{pi})-x(v_{0i})][x(v_{0j})-x(v_{pj})]\\
&+&2\sum_{i=t+1}^{k-1}[x(v_{pk})-x(v_{01})][x(v_{0i})-x(v_{pi})]\\
&>&2[x(v_{01})-x(v_{pk})]x(v_{p+1,1}) .
\end{eqnarray*}
Then
\begin{eqnarray*}
&&y^TA(G')y-x^TA(G)x\\
&>&2[x(v_{pk})-x(v_{01})]x(v_{p+1,1})+2[x(v_{01})-x(v_{pk})]x(v_{p+1,1})=0.
\end{eqnarray*}
Then $\lambda(G')\ge y^TA(G')y> x^TA(G)x=\lambda(G)$.

{\bf Subcase 2.3:} If $x(v_{p2})>x(v_{0,k-1})$. Let
 \begin{eqnarray*}
V_{01}''&=&\left(\bigcup_{i=1,...,k-1}N_G(v_{0i})\right) \Big{\backslash}\Big{\{}v_{0i}\Big{\}}_{i=1}^{k-1}\\
V_{0k}''&=&\{w|w\sim v_{0k},w\in V(G),w\notin \{v_{11},v_{01},...,v_{0,k-1}\}\}=N_G(v_{0k})\backslash \{v_{11},v_{01},...,v_{0,k-1}\}\\
E_1''&= &\{v_{pk}v_{p+1,1}\}\cup\{v_{0j}w\in E(G)|w\in V_{01}'',j=1,..,t\}\cup \{v_{0k}w|w\in V_{0k}''\}\cup\\
&&\{v_{0i}v_{0j}|1\le i<j\le k-1\}\cup\{v_{pi}v_{pj}|i\in \{2,3,..,k\},i\neq j\}\\
E_2''&=&\{uv\in E(G)| u,v \in \cup_{i=1}^{p-1}V_i\}\cup\{v_{0k}v_{11},v_{p-1,k}v_{p1}\}=E_2.
\end{eqnarray*}
Then it is divided into the following two subcases.

{\bf Subcase 2.3.1:} If $x(v_{0k})\ge x(v_{p1})$, then let $y$ be a vector defined by the following:
\begin{displaymath}
y(u) = \left\{ \begin{array}{ll}
x(v_{01}), & \textrm{ $u=v_{pk}$}\\
x(v_{pk}),& \textrm{ $u=v_{01}$}\\
x(v_{p-i,j}),& \textrm{ $u=v_{ij},i\in \{0,p\},j=1,2,...,k-1$}\\
x(u),& \textrm{ otherwise}
\end{array} \right..
\end{displaymath}
Then
 \begin{eqnarray*}
&&y^TA(G')y-x^TA(G)x\\
&=&y^TA(G')y-y^TA(G)y+y^TA(G)y-x^TA(G)x\\
&=&2[y(v_{01})y(v_{p+1,1})-y(v_{pk})y(v_{p+1,1})]+y^TA(G)y-x^TA(G)x\\
&=&2[x(v_{pk})x(v_{p+1,1})-x(v_{01})x(v_{p+1,1})]+y^TA(G)y-x^TA(G)x.
\end{eqnarray*}
By the definition of $ y $, it implies that $y(u)y(v) =x(u)x(v)$, if $uv\notin E_1''$. Then
\begin{eqnarray*}
&&y^TA(G)y-x^TA(G)x\\
&=&2\sum_{uv\in E_1''}[y(u)y(v)-x(u)x(v)]+2\sum_{uv\notin E_1''}[y(u)y(v)-x(u)x(v)]\\
&=&2[y(v_{pk})y(v_{p+1,1})-x(v_{pk})x(v_{p+1,1})]+2\sum_{j=1}^{k-1}\sum_{w\sim v_{0j},w\in V_{01}''}[y(v_{0j})y(w)-x(v_{0j})x(w)]\\
&+&\sum_{w\in V_{0k}''}2[y(v_{0k})y(w)-x(v_{0k})x(w)]+\sum_{1\le i<j\le k-1}2[y(v_{0i})y(v_{0j})-x(v_{0i})x(v_{0j})]\\
&+&\sum_{i\in \{2,3,..,k\}, j\notin \{2,3,..,k\}}2[y(v_{pi})y(v_{pj})-x(v_{pi})x(v_{pj})]\\
&+&\sum_{i,j\in \{2,3,..,k\},i< j}2[y(v_{pi})y(v_{pj})-x(v_{pi})x(v_{pj})]\\
&=&2[x(v_{01})-x(v_{pk})]x(v_{p+1,1})+2\sum_{j=2}^{k-1}\sum_{w\sim v_{0j},w\in V_{01}''}[y(v_{0j})y(w)-x(v_{0j})x(w)]\\
&+&\sum_{w\in V_{0k}''}2[x(v_{0k})x(w)-x(v_{0k})x(w)]+\sum_{2\le i<j\le k-1}2[x(v_{pi})x(v_{pj})-x(v_{0i})x(v_{0j})]\\
&+&\sum_{i=2}^{k-1}2[y(v_{pi})y(v_{p1})-x(v_{pi})x(v_{p1})]+\sum_{2\le i< j\le k-1}2[x(v_{0i})x(v_{0j})-x(v_{pi})x(v_{pj})]\\
&+&\sum_{w\in N_G(v_{01})}2[y(v_{01})y(w)-x(v_{01})x(w)]\\
&+&\sum_{w\in N_G(v_{pk})\backslash\{v_{p+1,1}\}}2[y(v_{pk})y(w)-x(v_{pk})x(w)]\\
&=&2[x(v_{01})-x(v_{pk})]x(v_{p+1,1})+2\sum_{i=2}^{k-1}\sum_{w\sim v_{0j},w\in V_{01}'}[x(v_{pi})x(w)-x(v_{0i})x(w)]\\
&+&\sum_{i=2}^{k-1}2[y(v_{pi})y(v_{p1})-x(v_{pi})x(v_{p1})]+\sum_{i=2}^{k-1}2[y(v_{01})y(v_{0i})-x(v_{01})x(v_{0i})]\\
&+&\sum_{i=2}^{k-1}2[y(v_{pk})y(v_{pi})-x(v_{pk})x(v_{pi})]\\
&+&\sum_{w\in N_G(v_{01})\backslash \{v_{02},...,v_{0,k-1}\}}2[y(v_{01})y(w)-x(v_{01})x(w)]+2[y(v_{pk})y(v_{p1})-x(v_{pk})x(v_{p1})]\\
&=&2[x(v_{01})-x(v_{pk})]x(v_{p+1,1})+2\sum_{i=2}^{k-1}\sum_{w\sim v_{0j},w\in V_{01}'}[x(v_{pi})-x(v_{0i})]x(w)\\
&+&\sum_{i=2}^{k-1}2[x(v_{0i})x(v_{p1})-x(v_{pi})x(v_{p1})]+\sum_{i=2}^{k-1}2[x(v_{pk})x(v_{pi})-x(v_{01})x(v_{0i})]\\
&+&\sum_{i=2}^{k-1}2[x(v_{01})x(v_{0i})-x(v_{pk})x(v_{pi})]\\
&+&\sum_{w\in N_G(v_{01})\backslash \{v_{02},...,v_{0,k-1}\}}2[x(v_{pk})x(w)-x(v_{01})x(w)]+2[y(v_{pk})y(v_{p1})-x(v_{pk})x(v_{p1})]\\
&=&2[x(v_{01})-x(v_{pk})]x(v_{p+1,1})+2\sum_{i=2}^{k-1}\sum_{w\sim v_{0j},w\in V_{01}'}[x(v_{pi})-x(v_{0i})]x(w)\\
&+&\sum_{i=2}^{k-1}2[x(v_{0i})-x(v_{pi})]x(v_{p1})+\sum_{w\in N_G(v_{01})\backslash \{v_{02},...,v_{0k}\}}2[x(v_{pk})-x(v_{01})]x(w)\\
&+&2[x(v_{pk})-x(v_{01})]x(v_{0k})+2[x(v_{01})-x(v_{pk})]x(v_{p1})
\end{eqnarray*}
Since $x(v_{p1})\le x(v_{0k}),x(v_{pk})>x(v_{01})$, $x(v_{0,k-1})< x(v_{p2})$ and $H\neq K_k$, then
\begin{eqnarray*}
&&y^TA(G)y-x^TA(G)x\\
&\ge&2[x(v_{01})-x(v_{pk})]x(v_{p+1,1})+2\sum_{i=2}^{k-1}\sum_{w\sim v_{0j},w\in V_{01}'}[x(v_{pi})-x(v_{0i})]x(w)\\
&+&\sum_{i=2}^{k-1}2[x(v_{0i})-x(v_{pi})]x(v_{p1})+2[x(v_{pk})-x(v_{01})][x(v_{0k})-x(v_{p1})]\\
&\ge&2[x(v_{01})-x(v_{pk})]x(v_{p+1,1})+\sum_{i=2}^{k-1}2[x(v_{pi})-x(v_{0i})][x(v_{0k})-x(v_{p1})]\\
&+&2[x(v_{pk})-x(v_{01})][x(v_{0k})-x(v_{p1})]\\
&\ge&2[x(v_{01})-x(v_{pk})]x(v_{p+1,1}),
\end{eqnarray*}
thus
 \begin{eqnarray*}
&&y^TA(G')y-x^TA(G)x\\
&=&2[y(v_{01})y(v_{p+1,1})-y(v_{pk})y(v_{p+1,1})]+y^TA(G)y-x^TA(G)x\\
&\ge&2[x(v_{pk})x(v_{p+1,1})-x(v_{01})x(v_{p+1,1})]+2[x(v_{01})-x(v_{pk})]x(v_{p+1,1})\\
&\ge& 0.
\end{eqnarray*}
So $\lambda(G')\ge y^TA(G')y\ge x^TA(G)x=\lambda(G)$ with equality if and only if $y$ is the Perron vector of $A(G')$. However,
\begin{eqnarray*}
(A(G')y)_{v_{p1}}&=&\sum_{j=2,...,k}y(v_{pj})=y(v_{pk})+\sum_{j=2,..,k-1}y(v_{pj})\\
&=&x(v_{01})+\sum_{j=2,..,k-1}x(v_{0j})<x(v_{pk})+\sum_{j=2,..,k-1}x(v_{pj})\\
&=&\lambda(G)x(v_{p1})=\lambda(G)y(v_{p1}).
\end{eqnarray*}
So $y$ is not the Perron vector of $A(G')$, thus $\lambda(G')>\lambda(G)$.

{\bf Subcase 2.3.2:} If $x(v_{0k})< x(v_{p1})$, then let $y$ be a vector defined by the following:
\begin{displaymath}
y(u) = \left\{ \begin{array}{ll}
x(v_{p-i,k+1-j}), & \textrm{ $u=v_{ij},i\in \{0,1,...,p\},j\in\{1,k\}$}\\
x(v_{p-i,j}), & \textrm{ $u=v_{ij},i=0,1,2,...,p;j=2,3,...,k-1$}\\
x(w),& \textrm{ otherwise}
\end{array} \right..
\end{displaymath}
Then
 \begin{eqnarray*}
&&y^TA(G')y-x^TA(G)x\\
&=&y^TA(G')y-y^TA(G)y+y^TA(G)y-x^TA(G)x\\
&=&2[y(v_{01})y(v_{p+1,1})-y(v_{pk})y(v_{p+1,1})]+y^TA(G)y-x^TA(G)x\\
&=&2[x(v_{pk})-x(v_{01})]x(v_{p+1,1})+y^TA(G)y-x^TA(G)x.
\end{eqnarray*}
For $uv\notin E_1''\cup E_2''$, $y(u)y(v)=x(u)x(v)$, by similar calculation in the {\bf Subcase 2.1.2}, then we can find that $2\sum_{uv\in E_2'}[y(u)y(v)-x(u)x(v)]=0$. Since $E_1''=\{v_{pk}v_{p+1,1}\}\cup E(G[V_1])\cup E(G[V_p])\cup \{v_{0i}w\in E(G)|w\notin \{v_{11},v_{01},...,v_{0k}\},i=1,..,k\}$. Let $W=\{v_{0i}w\in E(G)|w\notin \{v_{11},v_{01},...,v_{0k}\},i=1,..,k\}$. Thus
 \begin{eqnarray*}
&&y^TA(G)y-x^TA(G)x\\
&=&2\sum_{uv\in E_1'' \cup E_2''}[y(u)y(v)-x(u)x(v)]+2\sum_{uv\notin E_1'' \cup E_2''}[y(u)y(v)-x(u)x(v)]\\
&=&2[y(v_{pk})y(v_{p+1,1})-x(v_{pk})x(v_{p+1,1})]+2\sum_{i=1}^{k-1}\sum_{j=i+1}^k[y(v_{0i})y(v_{0j})-x(v_{0i})x(v_{0j})]\\
&+&2\sum_{i=1}^{k-1}\sum_{j=i+1}^k[y(v_{pi})y(v_{pj})-x(v_{pi})x(v_{pj})]+2\sum_{i=1}^{k}\sum_{w\sim v_{0i},w\in W}^k[y(v_{0i})y(w)-x(v_{0i})x(w)]\\
&=&2[x(v_{01})x(v_{p+1,1})-x(v_{pk})x(v_{p+1,1})]+2\sum_{1<i<j<k}[y(v_{0i})y(v_{0j})-x(v_{0i})x(v_{0j})]\\
&+&2\sum_{1<i<j<k}[y(v_{pi})y(v_{pj})-x(v_{pi})x(v_{pj})]+2\sum_{i=1}^{k}\sum_{w\sim v_{0i},w\in W}^k[x(v_{pi})x(w)-x(v_{0i})x(w)]\\
&+&2\sum_{1=i<j<k}[y(v_{01})y(v_{0j})-x(v_{01})x(v_{0j})]+2\sum_{1<i<j=k}[y(v_{0i})y(v_{0k})-x(v_{0i})x(v_{0k})]\\
&+&2\sum_{1=i<j<k}[y(v_{p1})y(v_{pj})-x(v_{p1})x(v_{pj})]+2\sum_{1<i<j=k}[y(v_{pi})y(v_{pk})-x(v_{pi})x(v_{pk})]\\
&=&2[x(v_{01})-x(v_{pk})]x(v_{p+1,1})+2\sum_{1<i<j<k}[x(v_{pi})x(v_{pj})-x(v_{0i})x(v_{0j})]\\
&+&2\sum_{1<i<j<k}[x(v_{0i})x(v_{0j})-x(v_{pi})x(v_{pj})]+2\sum_{i=1}^{k}\sum_{w\sim v_{0i},w\in W}^k[x(v_{pi})x(w)-x(v_{0i})x(w)]\\
&+&2\sum_{1=i<j<k}[x(v_{pk})x(v_{pj})-x(v_{01})x(v_{0j})]+2\sum_{1<i<j=k}[x(v_{pi})y(v_{p1})-x(v_{0i})x(v_{0k})]\\
&+&2\sum_{1=i<j<k}[y(v_{0k})x(v_{0j})-x(v_{p1})x(v_{pj})]+2\sum_{1<i<j=k}[x(v_{0i})x(v_{01})-x(v_{pi})x(v_{pk})]\\
&=&2[x(v_{01})-x(v_{pk})]x(v_{p+1,1})+2\sum_{i=1}^{k}\sum_{w\sim v_{0i},w\in W}^k[x(v_{pi})x(w)-x(v_{0i})x(w)].
\end{eqnarray*}
Since $H\neq K_k$, $x(v_{0k})< x(v_{p1})$, $x(v_{01})< x(v_{pk})$ and $x(v_{0,k-1})< x(v_{p2})$, then

$$y^TA(G)y-x^TA(G)x>2[x(v_{01})-x(v_{pk})]x(v_{p+1,1})$$.
Thus
\begin{eqnarray*}
&&y^TA(G')y-x^TA(G)x\\
&>&2[x(v_{pk})-x(v_{01})]x(v_{p+1,1})+2[x(v_{01})-x(v_{pk})]x(v_{p+1,1})=0.
\end{eqnarray*}
So $\lambda(G')\ge y^TA(G')y> x^TA(G)x=\lambda(G)$. This completes the proof.
\end{proof}

   Now we are ready to prove the Theorem ~\ref{T2}.

{\textbf {Proof.} } By Theorem \ref{Z}, $G\in \mathcal{T_{n,\alpha}}$. Next considering the following cases to prove the assertion:\\
{\bf Case 1:} If there are two vertices $u$ and $v$ each of which has at least two pendent clique paths adjacent with. Suppose that $u$ is adjacent with pendent clique paths $P_1$, $P_2$ and $v$ is adjacent with pendent clique paths $P_3$, $P_4$. Let $l_1$, $l_2$, $l_3$, $l_4$ be the lengths of the pendent clique paths $P_1$, $P_2$, $P_3$, $P_4$, respectively. Without loss of generality, let $l_1\geq l_3\geq l_4$. Then deleting the edge incident with $P_4$ and $u$, and adding it to the end of $P_3$ to get a new graph $G'$, By Lemma ~\ref{L2}, $\lambda(G)>\lambda(G')$, a contradiction with $\lambda(G)=\lambda_{n,\alpha}$.\\
{\bf Case 2:} If there is a vertex $u$ which has at least two pendent clique paths adjacent with, and there is not another vertex which has at least two pendent clique paths adjacent with. Suppose that $u$ is adjacent with pendent clique paths $P_1$, $P_2$,..., $P_t$. Assume $u$ is in the clique $G_1$ the size of which is $k$. By Lemma \ref{L6}, the degree of $V(G_1)\backslash \{u\}$ is $k-1$, suppose $v\in V(G_1)\backslash \{u\}$. Then delete some edge $uw$ which is not in $G_1$ and add edge $vw$ to get a new graph $G'$, and it is easy to find that $G'\in \mathcal{G}(n,\alpha)$. By Lemma \ref{L5}, $\lambda(G)>\lambda(G')$, Which contradicts with $\lambda(G)=\lambda_{n,\alpha}$. \\
{\bf Case 3:} If there is not a vertex which has at least two pendent clique paths adjacent with. By Lemma \ref{L6}, $G$ must be a clique path. \\
By the {\bf Case 1}, {\bf Case 2} and {\bf Case 3}, it can be found that the assertion holds.

\section{Proof of Theorem ~\ref{T3}}
\begin{proof}
  Let $G\in \mathcal{G}(n,\alpha)$, let $x$ be the Perron vector of $G$, then we consider the following cases:\\
{\bf Case 1:} There is a clique $G_1$ of order $k$ in $G$ which has two vertices $u$ and $v$ whose degrees are both larger than $k$. Without loss of generality, let $x(u)\geq x(v)$. Then deleting the edges incident with $w_{k}$ not in $G_1$ and adding them to $u$ to get a new graph $G'$, then $\lambda(G')\geq x^TA(G')x\geq x^TA(G)x=\lambda(G)$, by Rayleigh quotient principle, with equality holding if and only if $x$ is the eigenvector of $A(G')$. It is easy to find that $x$ is not the eigenvector of $A(G')$, so $\lambda(G')>\lambda(G)$.\\
{\bf Case 2:} For each clique $G_1$ of order $k$ in $G$, there is only one vertex in $G_1$ whose degree is larger than $k-1$. For any graph $H$, let $E_1(H)=\{e=uv\in E(H)|d(u)>k,d(v)>k\}$ and $N(G)=|E_1(H)|$. Let $uv\in E_1(G)$, without loss of generality, suppose $x(u)\geq x(v)$  and $vw_{k},vw_1,...,vv_t$ are all the edges which are not in any clique of order $k$. Then deleting the edges $vw_{k},vw_1,...,vv_t$ and adding edges $uw_{k},uw_1,...,uv_t$ to get a new graph $G'$, obviously $G'\in \mathcal{G}(n,\alpha)$, and $\lambda(G')\geq x^TA(G')x^T\geq x^TA(G)x^T=\lambda(G)$ with equality holding if and only if $x$ is an eigenvector of $A(G')$. It is easy to find that $x$ is not the eigenvector of $A(G')$, so $\lambda(G')>\lambda(G)$ and $N(G')<N(G)$.\\
Since $G\in \mathcal{G}(n,\alpha)$, then by {\bf Case 1} and {\bf Case 2}, it is easy to find that $\lambda(G)\leq \Lambda_{n,\alpha}(S(n,\alpha))$ with equality holding if and only if $G=S(n,\alpha)$.
\end{proof}

\end{document}